\documentclass[11pt, a4paper,leqno]{amsart}

\usepackage{amsmath,amsthm,amscd,amssymb,amsfonts, amsbsy}
\usepackage{latexsym}
\usepackage{txfonts}
\usepackage{exscale}
\usepackage{mathrsfs}

\usepackage{esint} 
\usepackage{enumerate}

\usepackage[colorlinks=true, linkcolor=blue, citecolor=red, urlcolor=red, backref=page]{hyperref} 
\usepackage{fancyhdr}

\usepackage[utf8]{inputenc}

\calclayout
\allowdisplaybreaks

\numberwithin{equation}{section}

\theoremstyle{plain}
\newtheorem{theorem}[equation]{Theorem}
\newtheorem{thm}[equation]{Theorem}

\newtheorem{lem}[equation]{Lemma}

\newtheorem{cor}[equation]{Corollary}

\theoremstyle{definition}

\newtheorem{ex}{Examples}

\newtheorem{exam}[equation]{Example}

\theoremstyle{remark}
\newtheorem{remark}[equation]{Remark}
\newtheorem{rem}[equation]{Remark}

\makeatletter
\@namedef{subjclassname@2020}{%
  \textup{2020} Mathematics Subject Classification}
\makeatother

\def\N{\mathbb {N}}
\def\R{\mathbb{R}}

\DeclareMathOperator{\supp}{supp}

\newcommand{\normmm}[1]{{\left\vert\kern-0.25ex\left\vert\kern-0.25ex\left\vert #1
		\right\vert\kern-0.25ex\right\vert\kern-0.25ex\right\vert}}

\def\Z{\mathbb{Z}}

\begin{document}

\title{A unified approach to self-improving property via $K$-functionals}

\author[Oscar Dom\'{i}nguez et al.]{Oscar Dom\'{i}nguez, Yinqin Li, Sergey Tikhonov,
Dachun Yang and Wen Yuan}

\address{(O. Dom\'{i}nguez) Departamento de An{\'a}lisis Matem{\'a}tico y Matem\'atica Aplicada\\ Facultad de
Matem{\'a}ticas\\ Universidad Complutense de Madrid\\ Plaza de
Ciencias 3, 28040 Madrid\\ Spain}
\email{oscar.dominguez@ucm.es}

\address{(Y. Li, D. Yang, W. Yuan) Laboratory of Mathematics and Complex
Systems (Ministry of Education of China),
School of Mathematical Sciences, Beijing Normal
University, Beijing 100875, The People's Republic of China}
\email{yinqli@mail.bnu.edu.cn, dcyang@bnu.edu.cn, wenyuan@bnu.edu.cn}

\address{(S. Tikhonov) Centre de Recerca Matem\`{a}tica, Campus de Bellaterra, Edifici C 08193 Bellaterra (Barcelona),
Spain; ICREA, Pg. Llu\'{i}s Companys 23, 08010 Barcelona, Spain, and Universitat Aut\`{o}noma de Barcelona,
Facultat de Ci\`{e}ncies, 08193 Bellaterra, Barcelona, Spain}
\email{stikhonov@crm.cat}
%
%

\thanks{\emph{Acknowledgements.} The authors thank Mario Milman and Petru Mironescu for a number of helpful suggestions. Oscar Dom\'{i}nguez is supported by Grupo UCM-970966, Yinqin Li, Dachun Yang, and
Wen Yuan are supported by the National Key Research and Development Program of China
(Grant No.\ 2020YFA0712900) and the National
Natural Science Foundation of China
(Grant Nos.\ 12371093, 12071197 and 12122102), and Sergey Tikhonov is supported
by PID2020-114948GB-I00, 2021 SGR 00087, AP 14870758,
the CERCA Programme of the Generalitat de Catalunya, and Severo Ochoa and Mar\'{i}a de Maeztu
Program for Centers and Units of Excellence in R\&D (CEX2020-001084-M)}



\begin{abstract}
	
	In this paper we obtain new quantitative estimates that improve the classical inequalities:  Poincar\'e--Ponce, Gaussian Sobolev,  and John--Nirenberg. Our method is based on the $K$-functionals and allows one to derive self-improving type inequalities.
We show the optimality of the method by obtaining 
  new Bourgain--Brezis--Mironescu and Maz'ya--Shaposhnikova limiting formulas. In particular, we derive these formulas for fractional powers 
  of infinitesimal generators of operator semigroups on 
    Banach spaces.

\end{abstract}

\maketitle

\setcounter{tocdepth}{1}


\section{Introduction}
The main goal of this paper is 
 to study self-improving properties of several celebrated inequalities in analysis.
We present a  general method to attack this problem  based on the \emph{$K$-functional} \cite{BennettSharpley, BerghLofstrom, DeVoreLorentz}. Recall that  the  $K$-functional relative to  a compatible pair\footnote{Loosely speaking, this means that $A_0 + A_1$ makes sense.} of normed spaces   $(A_0, A_1)$  is given by
$$
	K(t,f; A_0, A_1) := \inf_{f = f_0 + f_1} (\|f_0\|_{A_0} + t \, \|f_1\|_{A_1})$$
	for $t > 0$ and $f \in A_0 + A_1$.
		We start with  the following   result.

\begin{thm}\label{ThmIntPoincare}
	Assume that there exists a linear operator $T$ such that
	 \begin{equation}\label{Linearity}
	 T : A_0 \to  B \qquad \text{and} \qquad T: A_1 \to B,
	 \end{equation}
	 with corresponding norms $\|T\|_0 := \|T\|_{A_0 \to B}$ and $\|T\|_{1} := \|T\|_{A_1 \to B}$. Let $\{\rho_\varepsilon\}_{\varepsilon > 0}$ be a family of non-negative functions satisfying
	 	\begin{equation}\label{AssRho1}
		\int_0^\infty \rho_\varepsilon(t) \, dt = 1
	\end{equation}
	  with $\emph{supp } \rho_\varepsilon \subset \big(0, \frac{\|T\|_1}{\|T\|_0} \big)$ and let $p > 0$. Then
	\begin{equation}\label{PonceIntIneq}
		\|T f\|_{B} \leq  \|T\|_1\,  \left(\int_0^{\frac{\|T\|_1}{\|T\|_0}} \bigg(\frac{K(t, f; A_0, A_1)}{t}\bigg)^p \rho_\varepsilon (t ) \, dt \right)^{\frac{1}{p}}
	\end{equation}
	for every $f \in A_0 + A_1$ and $\varepsilon > 0$.
\end{thm}


We stress that Theorem \ref{ThmIntPoincare} is an interpolation-free statement in the sense that the right-hand side of \eqref{PonceIntIneq} does not involve standard interpolation norms. In fact, this result can be considered  an extension of  interpolation inequalities  to a non-interpolation setting (cf. Section \ref{Section:SelfImp} for a detailed discussion).  We provide a simple  proof of Theorem \ref{ThmIntPoincare} that relies exclusively on basic properties of $K$-functionals.  Note that \eqref{PonceIntIneq} provides a stronger assertion than the starting inequality:
\begin{equation}\label{AsymF}
\|T f\|_B \leq \|T\|_1 \|f\|_{A_1},
\end{equation}
cf. \eqref{Linearity}.
Indeed, it follows from $K(t, f; A_0, A_1) \leq t \|f\|_{A_1}$ and \eqref{AssRho1} that
$$
 \left(\int_0^{\frac{\|T\|_1}{\|T\|_0}} \bigg(\frac{K(t, f; A_0, A_1)}{t}\bigg)^p \rho_\varepsilon (t ) \, dt \right)^{\frac{1}{p}} \leq \|f\|_{A_1}, \qquad \text{for all} \quad \varepsilon > 0.
$$
Furthermore, under the additional assumption that, for every $a \in (0, 1)$,
	\begin{equation}\label{AssRho}
	\lim_{\varepsilon \to 0^+} \, \int_0^{a} \rho_\varepsilon (t) \, dt =1,
	\end{equation}
\eqref{AsymF} can be recovered as the limit of  \eqref{PonceIntIneq} as $\varepsilon \to 0^+$. Indeed, 
 we have\footnote{In \eqref{IntronewV3} we also assume a rather mild and natural condition on $(A_0, A_1)$, namely $(A_0, A_1)$ is Gagliardo closed; cf. Remark \ref{RemarkInt}.} (see Lemma \ref{LemmaReal} and Remark \ref{RemarkInt}) 
\begin{equation}\label{IntronewV3}
	\lim_{\varepsilon \to 0^+} \left(\int_0^{\frac{\|T\|_1}{\|T\|_0}} \bigg(\frac{K(t, f; A_0, A_1)}{t}\bigg)^p \rho_\varepsilon (t ) \, dt \right)^{\frac{1}{p}} = \|f\|_{A_1}.
\end{equation}
 Note that for the family $\displaystyle\rho_\varepsilon (t) = p(1-\varepsilon) \bigg(\frac{\|T\|_0}{\|T\|_1 } \, t \bigg)^{p(1-\varepsilon)} t^{-1} \mathbf{1}_{\big(0, \frac{\|T\|_1}{\|T\|_0} \big)},$ $\varepsilon \in (0, 1),$  \eqref{IntronewV3} gives Milman's extrapolation formula \cite[Theorem 1 and Lemma 3]{Milman}:
\begin{equation}\label{MEx}
	\lim_{\varepsilon \to 1^-}    \bigg((1-\varepsilon)  \int_0^{ \frac{\|T\|_1}{\|T\|_0}} [t^{-\varepsilon} K(t, f; A_0, A_1)]^p \frac{dt}{t} \bigg)^{\frac{1}{p}}= \frac{\|T\|_0}{p^{\frac{1}{p}} \|T\|_1 }  \, \|f\|_{A_1}.
\end{equation}

Although looking elementary, inequality \eqref{PonceIntIneq} provides a unifying approach to improve several important classical inequalities.
 In particular, we show how specific  choices of the spaces $A_0, A_1, B$ and the operator $T$ in Theorem \ref{ThmIntPoincare} lead to improvements of the following results:

\begin{itemize}
	\item[$\cdot$] Poincar\'e--Ponce inequality;
	\item[$\cdot$] Gaussian Sobolev inequality;
	\item[$\cdot$] John--Nirenberg inequality.
\end{itemize}
 Before going into more detail,
 we   recall the celebrated ``BBM-formula", by Bourgain, Brezis and Mironescu  \cite{BourgainBrezisMironescu} (see also \cite{Brezis} and  \cite[Chapter 6]{BrezisMironescu}), for the Sobolev norm and  the sharp version of Poincar\'e inequality due to Ponce \cite{Ponce}.

\subsection{Bourgain--Brezis--Mironescu formulas and Poincar\'e--Ponce inequalities}
  Assume that $\{\rho_\varepsilon\}_{\varepsilon > 0}$ is a family of non-negative functions on $(0, \infty)$ satisfying \eqref{AssRho1} and \eqref{AssRho}.
	Let $\Omega$ be a bounded smooth domain\footnote{The case $\Omega = \R^N$ is also admissible in \eqref{BBM}.} in $\R^N, \, N \geq 1,$ and $p \in [1, \infty)$. Then, for any $f \in W^{1, p}(\Omega)$,
	\begin{equation}\label{BBM}
		\lim_{\varepsilon \to 0^+} \, \int_\Omega \int_\Omega \frac{|f(x) -f(y)|^p}{|x-y|^{p+N-1}} \, \rho_\varepsilon(|x-y|) \, dx \, dy = C_{N, p} \, \int_\Omega |\nabla f(x)|^p \, dx,
	\end{equation}
	where
\begin{equation}\label{CBBM}
C_{N, p} :=   \int_{\mathbb{S}^{N-1}}
\left|\omega\cdot \mathbf{e}\right|^p
\,d\sigma^{N-1}(\omega).
\end{equation}
Here $\sigma^{N-1}$ is the surface Lebesgue measure on the
unit sphere $\mathbb{S}^{N-1}$ and $\mathbf{e} \in \mathbb{S}^{N-1}$ is any fixed vector.
 Furthermore\footnote{Given two non-negative quantities $A$ and $B$, the notation $A \lesssim B$ means that there exists a constant $C$, independent of all essential parameters, such that $A \leq C B$. We write $A \approx B$ if $A \lesssim B \lesssim A$.}
\begin{equation}\label{BBM2}
	\int_\Omega \int_\Omega \frac{|f(x) -f(y)|^p}{|x-y|^{p+N-1}} \, \rho_\varepsilon(|x-y|) \, dx \, dy \lesssim \int_\Omega |\nabla f(x)|^p \, dx
\end{equation}
for all $\varepsilon > 0$. For the special choice of $\rho_\varepsilon(t) = \frac{\varepsilon}{D^\varepsilon t^{1-\varepsilon}} \, \mathbf{1}_{(0, D)},$ where $D$ is the diameter of $\Omega$, \eqref{BBM} shows the continuity of the scale of fractional Sobolev spaces $\dot{W}^{s, p}(\Omega)$ (properly  renormalized) as $s \to 1^{-}$. More precisely,
 \begin{equation}\label{BBMClassic}
	\lim_{s \to 1^-} (1-s) \, \|f\|^p_{\dot{W}^{s, p}(\Omega)} = \frac{C_{N, p}}{p} \, \int_{\Omega} |\nabla f(x)|^p \, dx,
\end{equation}
where
$$
	  \|f\|_{\dot{W}^{s, p}(\Omega)} := \bigg(\int_\Omega \int_{\Omega} \frac{|f(x)-f(y)|^p}{|x-y|^{s p + N}}  \, dx \, dy \bigg)^{\frac{1}{p}}
$$
is the classical \emph{Gagliardo seminorm} of $f$. In particular, $\|f\|_{\dot{W}^{s, p}(\Omega)}$ blows up like $(1-s)^{-1/p}$ as $s \to 1^-$ unless $u$ is constant.

Formula \eqref{BBM} has  a crucial  impact on the field with far-reaching applications to PDE's and differential geometry and  has been extended to many different settings. For a recent account, 
 see the monograph by Brezis and Mironescu \cite{BrezisMironescu}. In particular, the counterpart of \eqref{BBMClassic}
 as $s \to 0^+$
  was obtained  by Maz'ya and Shaposhnikova \cite{Mazya}: if $f \in C^\infty_0(\R^N)$ then
\begin{equation}\label{MSClassic}
	\lim_{s \to 0^+} s \, \|f\|_{\dot{W}^{s, p}(\R^N)}^p = \frac{2 |\mathbb{S}^{N-1}|}{p} \, \int_{\R^N} |f(x)|^p \, dx.
\end{equation}
Analogs of \eqref{BBMClassic} and \eqref{MSClassic} for arbitrary integer smoothness were obtained in \cite{Milman} and \cite{KaradzhovMilmanXiao} as applications of interpolation techniques (cf. \eqref{MEx}). See also  \cite{Borghol, BIK, Ferreira}.

Motivated by the BBM-formula, Ponce \cite{Ponce} established a remarkable family of inequalities that improves the classical Poincar\'e inequality:
\begin{equation}\label{Poincare}
	\int_\Omega |f(x)-f_\Omega|^p \, dx \leq A_0 \,  \int_\Omega |\nabla f(x)|^p \, dx, \qquad p \in [1, \infty),
\end{equation}
where $f_\Omega := \frac{1}{|\Omega|} \int_\Omega f$. According to \cite[Theorem 1.1]{Ponce},  given any $N \geq 2$ and $\delta > 0$ there exists $\varepsilon_0 \in (0, 1)$ sufficiently small such that\footnote{In   \cite{Ponce} the discrete version of \eqref{PonceIneq} was considered,   under  the  transformation $\varepsilon \leftrightarrow \frac{1}{n}$, $n \in \N$.}
\begin{equation}\label{PonceIneq}
	\int_\Omega |f(x)-f_\Omega|^p \, dx \leq \bigg(\frac{A_0}{C_{N, p}} + \delta \bigg)  \, \int_\Omega \int_\Omega \frac{|f(x)-f(y)|^p}{|x-y|^{p+N-1}} \, \rho_\varepsilon(|x-y|) \, dx \,dy
\end{equation}
for every $f \in L^p(\Omega)$ and $\varepsilon \in (0, \varepsilon_0]$. In light of  \eqref{BBM2},  estimate  \eqref{PonceIneq} provides a stronger assertion than  \eqref{Poincare}. In fact, as a consequence of \eqref{BBM}, the exact constant in \eqref{Poincare} can  be recovered from  \eqref{PonceIneq} by taking limits as $\varepsilon \to 0^+$ and $\delta \to 0^+$.   The case $N=1$ is also covered by \cite[Theorem 1.3]{Ponce} but  additional assumptions on $\{\rho_\varepsilon\}$ are required.  {The proof of \eqref{PonceIneq} is by contradiction\footnote{This explains why the constant in \eqref{PonceIneq} is not explicit.} and relies on \eqref{BBM} together with a  BBM-compactness criterion in $L^p(\Omega)$ (see \cite[Theorem 1.2]{Ponce})}. A distinguished example of \eqref{PonceIneq} is  the classical BBM  inequality
\begin{equation}\label{Intro110}
	\int_\Omega |f(x)-f_\Omega|^p \, dx \lesssim (1-s)  \, \int_\Omega \int_\Omega \frac{|f(x)-f(y)|^p}{|x-y|^{s p + N}}  \, dx \,dy,
\end{equation}
for $0 < s_0 < s < 1$.

%
%

\subsection{Sharpened Poincar\'e--Ponce inequalities}\label{Section1.2}

As an  application of Theorem \ref{ThmIntPoincare},  we are able to sharpen and extend the Poincar\'e--Ponce inequality \eqref{PonceIneq} in several directions; see Theorem \ref{ThmSharpPPPhik} below.  To avoid unnecessary technicalities, 
   we focus on the case $\Omega = Q$, a cube in $\R^N$ with edges parallel to the axes of coordinates.
   Compared to the classical results, our accomplishments are as follows: 
\begin{enumerate}[{\rm (i)}]
\item We establish a strengthened version of \eqref{PonceIneq} with a remainder term. To be more precise, the family of inequalities  \eqref{PonceIneq} can be understood in terms of families of functionals $\{I_\varepsilon\}_{\varepsilon \in (0, 1)}$ such that
\begin{equation}\label{46new}
	\int_Q |f-f_Q|^p \leq C  I_\varepsilon(f)
\end{equation}
for some constant $C$ that does not depend on $\varepsilon$ and
$$
	\lim_{\varepsilon \to 0^{+}} I_\varepsilon(f) = C' \, \|\nabla f\|^p_{L^p(Q)}.
$$
Our results show that the right-hand side of \eqref{46new} can be improved: 
\begin{equation}\label{ImprIntro1}
	\int_Q |f-f_Q|^p \leq C  (I_\varepsilon(f) - J_\varepsilon(f))
\end{equation}
with  $J_\varepsilon(f) > 0$,
and
$$
	\lim_{\varepsilon \to 0^{+}} ( I_\varepsilon(f) - J_\varepsilon(f)) = C' \, \|\nabla f\|^p_{L^p(Q)}.
$$
The new family of inequalities \eqref{ImprIntro1} is a reminiscent of improved versions of classical Hardy--Sobolev inequalities involving remainder terms \cite{BrezisLieb, BrezisNirenberg}.
\item The case of higher-order derivatives is covered.
\item We can deal with functions in arbitrary r.i. spaces (rearrangement invariant Banach function spaces) on $\R^N, \, N \geq 1$. In particular, our formulation covers the delicate case $N=1$ without any further assumptions on $\{\rho_\varepsilon\}_{\varepsilon > 0}$. This is in sharp contrast with  \eqref{PonceIneq}.

\item Our estimates hold for all $\varepsilon > 0$ with underlying equivalence constants independent of any auxiliary parameter $\delta > 0$, which is not the case in \eqref{PonceIneq}. In this respect, we assert that \eqref{SharpPPk} below is a more explicit estimate than \eqref{PonceIneq}.

\item  The family $\{\rho_\varepsilon\}_{\varepsilon > 0}$ does not have to  satisfy  condition \eqref{AssRho}, but only the very mild assumption \eqref{AssRho1}.
\end{enumerate}

Before giving the precise statement of sharpened Poincar\'e--Ponce inequalities, we briefly discuss how these inequalities can be connected with $K$-functionals in such a way that Theorem \ref{ThmIntPoincare} can be applied. The connection is made through approximation techniques, namely,  Whitney inequalities for the best approximation error. Recall that, for $k \in \N$ and $p \in [1, \infty]$, the \emph{best approximation} of any $f \in L^p(Q)$ with respect to $\mathcal{P}_{k-1}$, the space of polynomials of degree at most $k-1$, is measured by
\begin{equation}\label{localerror}
	E_{k}(f, Q)_p := \inf_{P\in\mathcal{P}_{k-1}}\left\|f-P\right\|_{L^p(Q)}.
\end{equation}
Observe that
\begin{equation}\label{Means}
E_1(f, Q)_p \approx \bigg(\int_Q |f-f_Q|^p \bigg)^{1/p}.
\end{equation}
This concept plays a prominent role in approximation theory, cf. the survey  by Brudnyi \cite{Brudnyi} and references therein. In particular, a central issue in approximation theory is the  \emph{Whitney-type inequalities}, which establish bounds for $E_{k}(f, Q)_p$ in terms of smoothness characteristics of $f$ usually given by the moduli of smoothness $\omega_k(f, t)_{L^p(Q)}$ (cf. \eqref{DefMod}). The classical Whitney inequality states that
\begin{equation}\label{WhitneyEstim}
	E_{k}(f, Q)_p \lesssim  \omega_k(f, \ell(Q))_{L^p(Q)},
\end{equation}
where $\ell(Q)$ is the edge length of $Q$. This result has a long and rich history, which goes back to Whitney \cite{Whitney} if $N=1$ and $p=\infty$. There are further extensions to  $p \in (0, \infty]$, more general domains (convex and Lipschitz domains), anisotropic versions in terms of parallelepipeds, etc.
In particular, we are interested in the extension of \eqref{WhitneyEstim} to  general function spaces (say, r.i. spaces) $X$ on $\R^N$  
  (cf. \cite[Theorem 1]{Brudnyi70})
\begin{equation}\label{WhitneyEstimB}
	E_{k}(f, Q)_X \leq C  \omega_k(f, \ell(Q))_{X(Q)},
\end{equation}
where\footnote{The definition of $E_{k}(f, Q)_X$ is the usual modification of \eqref{localerror}, where $L^p(Q)$ is replaced by $X(Q)$ (cf. \eqref{LocalX}).} $C$ depends only on $k$ and $N$.  Using the known estimates of moduli smoothness, we obtain
\begin{equation}\label{4.4}
	\omega_k(f, t)_{X(Q)} \lesssim t^k \, \|\nabla^k f\|_{X(Q)}.
\end{equation}
It is clear that \eqref{WhitneyEstimB} implies the following (higher order) Poincar\'e inequality:
\begin{equation}\label{ThmSharpPPPhike}
	E_{k}(f, Q)_X \lesssim \ell(Q)^k \, \|\nabla^k f\|_{X(Q)}.
\end{equation}	
In particular, the classical inequality \eqref{Poincare} corresponds to $k=1$ and $X = L^p(\R^N)$ (cf. \eqref{Means}).

Further, we establish the desired link between the Poincar\'e inequality and  $K$-functionals.
 Taking into account the well-known equivalence (with equivalence constants independent of $t, f$ and $Q$)
\begin{equation}\label{EstimKFunctMod}
	K(t^k, f; X(Q), \dot{W}^{k} X (Q)) \approx \omega_k (f, t)_{X(Q)}
\end{equation}
(cf. \cite[pp. 339--341]{BennettSharpley} and \cite[Chapter 6, Theorem 2.4]{DeVoreLorentz} for $X = L^p$, but similar arguments work with any r.i. space $X$),
enables us to
improve the  Poincar\'e inequality  \eqref{ThmSharpPPPhike} as follows:
\begin{equation}\label{Bridge}
	E_{k}(f, Q)_X \lesssim K(\ell(Q)^k, f; X(Q), \dot{W}^{k} X(Q)).
\end{equation}
  Applying our general
   self-improving technique
  to \eqref{Bridge}, we establish the following Poincar\'e--Ponce-type inequalities. The result has a simple and elegant formulation even in the  general setting of r.i. spaces and higher order derivatives, and it is new even in the prototypical case $X=L^p(\R^N)$ and $k=1$ (cf. Remark \ref{Remark1.28}).

\begin{thm}\label{ThmSharpPPPhik}
Let $X$ be a r.i. space on $\R^N$ and let  $p \in (0, \infty), \, k\in\mathbb{N}$.
Given any cube $Q$, consider $\{\rho_\varepsilon\}_{\varepsilon > 0}$ satisfying \eqref{AssRho1}
with $\emph{supp } \rho_\varepsilon \subset (0, \ell(Q)^k)$ and
define
\begin{equation}\label{gsfg}
	\phi_{\varepsilon} (t)
:= \int_{t}^{\ell(Q)} u^{k(1-p)-N}
\rho_\varepsilon(u^k) \, \frac{du}{u}, \qquad t \in (0, \ell(Q)).
\end{equation}
Then, for any $f\in X(Q)$
and $\varepsilon\in(0,\infty)$,
	\begin{equation}\label{SharpPPk}
		E_{k}(f, Q)_X^p	\lesssim \ell(Q)^{kp}
\int_{|h| \leq \ell(Q)}  \left\|\Delta^k_h f \right\|_{X(Q(k, h))}^p
\phi_{\varepsilon} (|h|) \,dh,
	\end{equation}
	where $Q(k, h) := \{x \in Q: x + k h \in Q\}.$
\end{thm}

\begin{remark}[The case $X = L^p$]\label{Remark1.28}
 One can apply Fubini's theorem in \eqref{SharpPPk}  to establish
	\begin{equation}\label{SharpPPkLp}
		E_{k}(f, Q)_p^p	\lesssim \ell(Q)^{kp}
\int_{Q}  \int_{\widetilde{Q}(k,x)}
\left|\Delta^k_hf(x)\right|^p
\phi_{\varepsilon} (|h|) \,dh\,dx,
	\end{equation}
	where $\widetilde{Q}(k, x) := \{h \in \R^N : x + k h \in Q\}$ for $x \in Q.$ Letting $k=1$ in the previous estimate, we have
	\begin{equation}\label{SharpPPkLp2}
		\int_Q |f-f_Q|^p \lesssim \ell(Q)^p \int_Q \int_Q |f(x)-f(y)|^p \phi_\varepsilon(|x-y|) \, dx \, dy
	\end{equation}
	where (cf. \eqref{gsfg})
	\begin{equation}\label{SharpPPkLp3}
	\phi_{\varepsilon} (t)
= \int_{t}^{\ell(Q)}
\frac{\rho_\varepsilon(u)}{u^{p+N}} \, du.
\end{equation}
A quick comparison between  \eqref{SharpPPkLp2} and \eqref{PonceIneq} shows that the weights $\frac{\rho_\varepsilon(|x-y|)}{|x-y|^{p+N-1}}$ in the $L^p$-oscillations of $f$ in  \eqref{PonceIneq} are replaced by $\phi_\varepsilon(|x-y|)$ in \eqref{SharpPPkLp2}. The fact that $\phi_\varepsilon$ is defined as an integral of type $\int_t^{\ell(Q)}$ (rather than $\int_t^\infty$) supplies us with the desired remainder term in Poincar\'e--Ponce inequalities. This will become apparent while working with particular choices of $\rho_\varepsilon$.
\end{remark}

\begin{remark}
We claim that  \eqref{SharpPPk} consists of an improvement of the Poincar\'e inequality \eqref{ThmSharpPPPhike} (in particular, \eqref{SharpPPkLp} sharpens \eqref{Poincare}). Indeed, we have
	\begin{equation}\label{SobSharp}
	 \ell(Q)^{kp}
\int_{|h| \leq \ell(Q)}  \left\|\Delta^k_h f \right\|_{X(Q(k, h))}^p
\phi_{\varepsilon} (|h|) \,dh \lesssim \ell(Q)^{k p} \,  \|\nabla^k f\|^p_{X(Q)}
	\end{equation}
	for every $\varepsilon > 0$. To see this, we can use $\|\Delta^k_h f\|_{X(Q(k, h))} \lesssim |h|^k \|\nabla^k f\|_{X(Q)}$,  \eqref{gsfg}, Fubini's theorem, a simple change of variables and \eqref{AssRho1} to get
	\begin{align*}
		\int_{|h| \leq \ell(Q)}  \left\|\Delta^k_h f \right\|_{X(Q(k, h))}^p
\phi_{\varepsilon} (|h|) \,dh &\lesssim   \big\|\nabla^k f \big\|_{X(Q)}^p \int_0^{\ell(Q)}  \phi_{\varepsilon} (t) t^{N-1} \,  dt \\
& \approx    \big\|\nabla^k f \big\|_{X(Q)}^p  \int_0^{\ell(Q)^k} \rho_\varepsilon(t) \, dt =  \big\|\nabla^k f \big\|_{X(Q)}^p.
	\end{align*}
	 Furthermore, we will show in Section \ref{SubSection6.1} (specifically, the local version of Theorem \ref{ThmGenSobk}) that (under the additional assumption \eqref{AssRho} and $X$  has an absolutely continuous norm)
	$$
		\lim_{\varepsilon \to 0^+}  	\int_{|h| \leq \ell(Q)}  \left\|\Delta^k_h f \right\|_{X(Q(k, h))}^p
\phi_\varepsilon (|h|) \,dh \approx   \|\nabla^k f\|^p_{X(Q)}.
	$$
	Therefore, \eqref{ThmSharpPPPhike} (respectively,  \eqref{Poincare}) can be obtained from \eqref{SharpPPk} (respectively, \eqref{SharpPPkLp}) by taking limits as $\varepsilon \to 0^+$.
\end{remark}

 The proof of Theorem \ref{ThmSharpPPPhik} is presented in Section \ref{Subsection:StaDisc}.
In Section \ref{SectionCompare}, we show that the theorem can be   easily  implemented for standard choices of $\{\rho_\varepsilon\}_{\varepsilon > 0}$. Then we compare the  outcomes with the corresponding results from \cite{Ponce} and observe that Theorem \ref{ThmSharpPPPhik} yields sharper estimates than \cite{Ponce}.

\subsection{Gaussian Sobolev inequalities}\label{Section1.3}

Let $(\R^N, \gamma_N)$ be the \emph{Gauss space}, i.e., $\R^N$ equipped with the Gauss measure
$
	d \gamma_N (x) = (2 \pi)^{-\frac{N}{2}} \, e^{-\frac{|x|^2}{2}} dx.
$
We consider the \emph{Ornstein-Uhlenbeck operator}
\begin{equation}\label{OUIntro}
	\mathcal{L} u := \Delta u - x \cdot \nabla u,
\end{equation}
which  generates the \emph{Ornstein-Uhlenbeck semigroup}
\begin{equation}\label{OUIntro2}
	G_t f (x) :=  (1-e^{-2 t})^{-N/2} \int_{\R^N} e^{-\frac{e^{-2 t} (|x|^2 + |y|^2 -2 x \cdot y)}{1-e^{-2 t}}} f(y) \, d \gamma_N(y).
\end{equation}

Logarithmic Sobolev inequalities have a long history, which goes back to the seminal paper  by Gross \cite{Gross}. In particular,
 it is known that
   there exists  a constant $C$,  independent of $N$, such that
\begin{equation}\label{GaussSobolev}
	\|f-m(f)\|_{L^p (\log L)^p(\R^N, \gamma_N)} \leq C \,  \|\mathcal{L} u\|_{L^p(\R^N, \gamma_N)}, \qquad p \in (1, \infty).
\end{equation}
Here, $m(f) := \int_{\R^N} f \, d \gamma_N$  and $L^p (\log L)^p(\R^N, \gamma_N)$ is the \emph{Zygmund space} equipped with
$$
	\|f\|_{L^p (\log L)^p(\R^N, \gamma_N)} :=  \bigg( \int_0^{1} [(1-\log t) \, f^{*}_{\gamma_N}(t)]^p \, dt \bigg)^{\frac{1}{p}}.
$$
As usual, $f^*_{\gamma_N}$ is the \emph{non-increasing rearrangement} of $f$ with respect to $\gamma_N$.
We do not comment here on the far-reaching applications of logarithmic Sobolev inequalities, instead we refer the reader to \cite{Bogachev} and \cite{Ledoux}.

A natural  interesting  question is to establish  Ponce's analogues of \eqref{GaussSobolev}, i.e.,  quantitative logarithmic Sobolev inequalities  that converge to \eqref{GaussSobolev}.
Here we note  that
 many BBM-type results admit natural extensions to the setting of metric measure spaces. However, these arguments strongly rely  on the fact that the underlying measure satisfies the doubling condition, which is not the case for the Gaussian measure.
 In this paper, we  overcome these obstructions to obtain the following result.

\begin{thm}\label{ThmGSL}
	Let $p \in (1, \infty)$ and let $\{\rho_\varepsilon\}_{\varepsilon > 0}$ be any family of non-negative functions with $\emph{supp } \rho_\varepsilon \subset (0, 1)$. Define
	\begin{equation}\label{UpsilonDef}
		\Upsilon_\varepsilon (t) :=   (1-\log t)^p \int_0^{(1-\log t)^{-1}} \rho_\varepsilon(u) \, du + \int_{(1-\log t)^{-1}}^{1} \frac{\rho_\varepsilon(u)}{u^p} \, du, \qquad t \in (0, 1).
	\end{equation}
	Then there exists a positive constant $C$, which is independent of  $\varepsilon > 0$ and $N$,  such that
	\begin{equation}\label{ThmGSL1eps}
			\int_0^{1}  [f-m(f)]^{* p}_{\gamma_N}(t) \, \Upsilon_\varepsilon(t) \, dt \leq C   \int_{\R^N} \int_0^1 \frac{|f(x)-G_t f(x)|^p}{t^p}
 \, \rho_\varepsilon(t) \, dt \, d\gamma_N(x)
	\end{equation}
	for every $f \in L^p(\R^N, \gamma_N)$.
\end{thm}

We note that $\Upsilon_\varepsilon (t)$ is decreasing on $(0, 1)$ but
$\Upsilon_\varepsilon (t)(1-\log t)^{-p}$ is increasing.
The proof of Theorem \ref{ThmGSL} is given in Section \ref{ProofG}.  Again, it relies  on the self-improving method given in Theorem \ref{ThmIntPoincare}  together with some limiting interpolation techniques from \cite{DominguezTikhonov}. In Section \ref{Section42} (see Theorem \ref{Theorem49}), we show that Theorem \ref{ThmGSL} improves the classical logarithmic Sobolev inequality \eqref{GaussSobolev}, in the sense that the latter follows from \eqref{ThmGSL1eps} by taking limits as $\varepsilon \to 0^+$.

Several applications of Theorem \ref{ThmGSL} are given  in Section \ref{Section43}. For example,  the special choices $\rho_\varepsilon(t) = \varepsilon t^{\varepsilon-1} \mathbf{1}_{(0, 1)}(t)$  and  $\rho_\varepsilon (t) = \varepsilon t^{p-\varepsilon-1} \mathbf{1}_{(0, 1)}(t)$ yield (cf. Corollary \ref{CorGauss1})
 \begin{equation}\label{Intro129}
	\int_0^{1} \{ (1-\log t)^{\varepsilon} [f-m(f)]^*_{\gamma_N} (t)\}^p \, dt \lesssim   \varepsilon (1- \varepsilon) \, \int_{\R^N} \int_0^1 \frac{|f(x)-G_t f(x)|^p}{t^{ \varepsilon p}} \, \frac{dt}{t} \, d \gamma_N(x)
	\end{equation}
	for all $\varepsilon \in (0, 1)$. In particular,
		 \begin{equation}\label{Intro130}
	\|f-m(f) \|_{L^p(\R^N, \gamma_N)}^p \lesssim   \varepsilon (1- \varepsilon) \, \int_{\R^N} \int_0^1 \frac{|f(x)-G_t f(x)|^p}{t^{ \varepsilon p}} \, \frac{dt}{t} \, d \gamma_N(x).
	\end{equation}
Note that the double integral on the right-hand side of \eqref{Intro129}  and \eqref{Intro130} defines a classical Gaussian Besov space\footnote{As usual in Gaussian analysis, the classical differences $|f(x)-f(y)|$ in $\R^N$ are  replaced by $|f(x)-G_t f(x)|$ in $(\R^N, \gamma_N)$.}. Interestinlgy, \eqref{Intro129} exhibits a new phenomenon that is not true in the BBM inequality \eqref{Intro110}: the gain $\varepsilon (1-\varepsilon)$ as $\varepsilon \to 0^+$ and $\varepsilon \to 1^-$. The explanation why the prefactor $\varepsilon$ appears in \eqref{Intro129} is that Theorem \ref{ThmGSL} holds without any further assumption on $\{\rho_\varepsilon\}_{\varepsilon > 0}$, while Theorem \ref{ThmSharpPPPhik} requires the normalization condition  \eqref{AssRho1} (which
is not true for $\rho_\varepsilon (t) = \varepsilon t^{p-\varepsilon-1} \mathbf{1}_{(0, 1)}(t)$).
	
We mention that  a complete  treatment (without the BBM phenomenon) on embeddings for Gaussian Besov spaces can be found in Mart\'in and Milman \cite[Chapter 6]{MartinMilman}.
Another approach has been recently proposed by Bogachev, Kosov, and Popova \cite{Bogachev19, Kosov}. In particular, \eqref{Intro129}  answers affirmatively the question raised by Kosov in \cite{Kosov}; see the discussion after Corollary \ref{CorGauss1}.
	
	\subsection{John--Nirenberg inequalities sharpened}\label{Section1.4}
	We start by recalling the definition of $\text{BMO}(\R^N)$ ($=$ \emph{bounded mean oscillation}) of John and Nirenberg \cite{JohnNirenberg}. We say that a locally integrable function $f$ belongs to $\text{BMO}(\R^N)$ if
\begin{equation}\label{BMODef}
	\|f\|_{\text{BMO}(\R^N)} = \|f^{\#} \|_{L^\infty(\R^N)} < \infty,
\end{equation}
where $f^{\#}$ is the \emph{Fefferman--Stein maximal function} defined by
\begin{equation}\label{SMax}
	f^{\#} (x) :=  \sup_{Q \ni x} \,  \fint_Q |f-f_Q|, \qquad x \in \R^N.
\end{equation}

	The celebrated John--Nirenberg inequality for $\text{BMO}(\R^N)$ asserts that\footnote{Using the well-known fact that $e^L$, the Orlicz space of exponentially integrable functions, can be characterized as
$$
	\|f\|_{e^L(Q)} \approx \sup_{p > 1} \frac{\|f\|_{L^p(Q, \frac{dx}{|Q|})}}{p},
$$
the assertion \eqref{JN} can be rephrased in terms of (local) embeddings from $\text{BMO}(\R^N)$  to $e^L(Q)$.} there exists a purely dimensional constant $c_N$ such that, for any $p < \infty$,
\begin{equation}\label{JN}
	\bigg(\fint_Q |f-f_Q|^p \bigg)^{1/p} \leq c_N  p \, \|f\|_{\text{BMO}(\R^N)}
\end{equation}
for all cubes $Q$.

In this paper, we  show that \eqref{JN} can be sharpened using Theorem \ref{ThmIntPoincare}. Specifically, we obtain the following results.

\begin{thm}\label{ThmJNSharp}
Let $p \in  [1, \infty)$ and let $\{\rho_\varepsilon\}_{\varepsilon > 0}$ be a family of functions satisfying \eqref{AssRho1} with $\emph{supp } \rho_\varepsilon \subset (0, 1)$.  Define
\begin{equation}\label{ThmJNSharp1}
	\eta_{\varepsilon, p} (t) = \int_{t^{\frac{1}{p}}}^1 \frac{\rho_\varepsilon(u) }{u^p} \, du, \qquad \forall\, t \in (0, 1).
\end{equation}
Therefore there exists a constant $C_N$, depending only on $N$, such that
\begin{equation}\label{ThmJNSharp2}
	\bigg(\fint_Q |f-f_Q|^p \bigg)^{\frac{1}{p}} \leq C_N p \, \bigg(\frac{1}{|Q|} \, \int_0^{|Q|} f^{\#*}(t)^p \, \eta_{\varepsilon, p} \Big(\frac{t}{|Q|} \Big) \, dt \bigg)^{\frac{1}{p}}
	\end{equation}
	for every $f \in L^p(\R^N), \varepsilon > 0,$ and $Q$.
\end{thm}

 Estimate \eqref{ThmJNSharp2} improves the classical John--Nirenberg inequality \eqref{JN} as stated in the next  assertion.

\begin{thm}\label{ThmExtrJN}
	Let $p \in [1, \infty)$ and let $\{\rho_\varepsilon\}_{\varepsilon > 0}$ be a family of functions satisfying \eqref{AssRho1} with $\emph{supp } \rho_\varepsilon \subset (0, 1)$ and related family $\{ \eta_{\varepsilon, p} \}_{\varepsilon > 0}$ (cf. \eqref{ThmJNSharp1}).
Therefore
\begin{equation}\label{ThmExtrJN1}
	  \bigg(\frac{1}{|Q|} \, \int_0^{|Q|} f^{\#*}(t)^p \, \eta_{\varepsilon, p} \Big(\frac{t}{|Q|} \Big) \, dt \bigg)^{\frac{1}{p}} \leq \|f\|_{\emph{BMO}(\R^N)}
\end{equation}
for every $\varepsilon > 0$ and $Q$. If, in addition,  $\{\rho_\varepsilon\}_{\varepsilon > 0}$ also satisfies \eqref{AssRho}, then
\begin{equation}\label{ThmExtrJN2}
	\lim_{\varepsilon \to 0^+} \,   \bigg(\frac{1}{|Q|} \, \int_0^{|Q|} f^{\#*}(t)^p \, \eta_{\varepsilon, p} \Big(\frac{t}{|Q|} \Big) \, dt \bigg)^{\frac{1}{p}}  = \|f\|_{\emph{BMO}(\R^N)}.
 \end{equation}
\end{thm}

	The proofs of Theorems \ref{ThmJNSharp} and \ref{ThmExtrJN} are given in Sections \ref{Section51} and \ref{Section52}, respectively. In Section \ref{Section53} we give several examples of \eqref{ThmJNSharp2} for special choices of $\{\rho_\varepsilon\}_{\varepsilon > 0}$.
	
	\subsection{New limiting formulas}
	
 In order to show that Theorem \ref{ThmSharpPPPhik} improves the classical Poincar\'e inequalities,
   we need to obtain BBM-type formulas (cf. \eqref{BBM}) related to the families of functionals
		\begin{equation}\label{142new}
		\int_{\R^N} \left\|\Delta_h^k f\right\|_X^p
\, \phi_\varepsilon(|h|) \, dh, \qquad \varepsilon > 0,
	\end{equation}
	where $X$ is a r.i. space and $\phi_\varepsilon$ is given by \eqref{gsfg}. This question is interesting by its own sake and is also motivated by  the recent results in \cite{DGPYYZ}, where a similar problem was considered for certain variants of \eqref{142new} with $k=1$ (cf. the discussion in Remark \ref{Remark627}). Accordingly,
we obtain the  required BBM-type formula (see Section \ref{SubSection6.1})
	\begin{equation}\label{143new}
		\lim_{\varepsilon \to 0^+}
\int_{\R^N} \left\|\Delta_h^k f\right\|_X^p
\, \phi_\varepsilon(|h|) \, dh  \approx
\left\|\nabla^k f\right\|_X^p.
	\end{equation}
	In fact we get a stronger statement with explicit constants.
	
The next natural step is  to obtain the Maz'ya--Shaposhnikova counterparts for \eqref{143new}, i.e., limiting formulas of type  \eqref{143new} for $\left\|f\right\|_X^p$ rather than $\left\|\nabla^k f\right\|_X^p$; see \eqref{MSClassic}. To do this, we  introduce the family of weights $\{\varphi_\varepsilon\}_{\varepsilon > 0}$  playing the role of $\{\phi_\varepsilon\}_{\varepsilon > 0}$ in \eqref{143new}.  Assume first that $\{\psi_\varepsilon\}_{\varepsilon > 0}$ is a family of non-negative functions on $(0, \infty)$ satisfying
	\begin{equation}\label{AssPsi1}
		\int_0^\infty \psi_\varepsilon(t) \, dt = 1
	\end{equation}
	and, for every $a > 1$,
	\begin{equation}\label{AssPsi}
	\lim_{\varepsilon \to 0^+} \, \int_{a}^\infty \psi_\varepsilon (t) \, dt = 1.
	\end{equation}
	Examples of these families are collected in Appendix \ref{SectionAA} (cf. Examples \ref{Ex1Dual}). Then (cf. Theorem \ref{ThmGenSobMSk})
	\begin{equation}\label{144new}
		\lim\limits_{\varepsilon\to0^+}
\int_{\R^N}\left\|\Delta^k_h f\right\|_X^p
\varphi_{\varepsilon}(|h|)\,dh
\approx
\left\| f\right\|_X^p,
	\end{equation}
	where $\varphi_\varepsilon$ is given by
	$$
	\varphi_{\varepsilon} (t) = \int_{t^k}^\infty u^{-\frac{N}{k}}
\psi_\varepsilon(u) \, du.
$$
In fact, the exact constant in \eqref{144new} is achieved and can be easily computed for special cases of $X$ (e.g. Lorentz spaces, see Example \ref{3.31}).
To the best of our knowledge, this result is new even for the Lebesgue spaces $X = L^p(\R^N)$ and extends the classical formula \eqref{MSClassic}, which refers to  the choice $\psi_\varepsilon(t) = \varepsilon p t^{-\varepsilon p-1}
{\bf 1}_{(1,\infty)}(t)$ (and $k=1$), to arbitrary families $\{\psi_\varepsilon\}_{\varepsilon > 0}$ with \eqref{AssPsi1}-\eqref{AssPsi}.

Several examples of the formulas \eqref{143new} and \eqref{144new} are collected in Section \ref{Section6.2}. On the other hand,  Section \ref{Section63} refers to an alternative methodology to limiting formulas, where
  the classical differences $\Delta^k_h f$ in \eqref{143new} and \eqref{144new} are now replaced by their averages
$$
	\bigg(\frac{1}{t^{\frac{N}{k}}} \int_{|h| \leq t^{\frac{1}{k}}} |\Delta_h^k f (x)|^p \, dh \bigg)^{\frac{1}{p}},
$$
as frequently used in approximation theory \cite{KolomoitsevTikhonov}.

{A far-reaching extension of \eqref{143new} and \eqref{144new} to the abstract setting of semigroups $\{T_t\}_{t > 0}$  on a  Banach space $X$ is proposed in Section \ref{Section6.4}.  In particular, our method allows us to replace $\nabla^k f$ in BBM formulas by  fractional powers $(-\mathcal{A})^\alpha f, \, \alpha > 0,$ of the corresponding infinitesimal generator $\mathcal{A}$.
This approach
is also motivated by the above mentioned assertion that the classical Gaussian Sobolev inequality \eqref{GaussSobolev} follows from the family of estimates \eqref{ThmGSL1eps} involving the Ornstein-Uhlenbeck semigroup. Our result (see Theorem \ref{ThmSemGen}) states
that,
 under the conditions \eqref{AssRho1} and  \eqref{AssRho}, 
\begin{equation}\label{147new2}
	\lim_{\varepsilon \to 0^+}  \int_0^\infty  \frac{\|[I-T_t]^\alpha f\|^p_X}{t^{\alpha(p-1)}} \,  \rho_\varepsilon (t^\alpha) \, \frac{dt}{t}  = \frac{1}{\alpha} \,  \|(-\mathcal{A})^\alpha f\|^p_{X},
\end{equation}
where $p > 0$ and  $[I-T_t]^\alpha$ is the fractional difference of order $\alpha>0$ (see \eqref{DefFG2}).
Taking $T_t f (x) = f(x+t)$ for $f \in X = L^p(\R)$, we are able to extend the BBM formula \eqref{BBM} from $\nabla f$ to the fractional Laplacian $(-\Delta)^{\alpha/2} f$, see Remark \ref{Remark659}. On the other hand, the choice $T_t = G_t$ (cf. \eqref{OUIntro2}) on $X = L^p(\R^N, \gamma_N)$ and $\alpha =1$ yields (cf. \eqref{OUIntro})
$$
	\lim_{\varepsilon \to 0^+}\, 	 \int_{\R^N} \int_0^1 \frac{|f(x)-G_t f(x)|^p}{t^p}
 \, \rho_\varepsilon(t) \, dt \, d\gamma_N(x) =  \|\mathcal{L} u\|_{L^p(\R^N, \gamma_N)}^p.
$$
 The analogue of \eqref{147new2} for 
   $\|f\|_X$ in place of
   $\|(-\mathcal{A})^\alpha f\|_{X}$ is also obtained.}

We conclude  the paper with two appendices. In Appendix \ref{SectionAA} we collect some standard examples of operators approximating unity that are used repeatedly. In Appendix  \ref{AppendixB} we compute the $K$-functional associated with $(X, \dot{W}^k X)$ in terms of the averaged moduli of smoothness. This result plays a key role in the proof of Theorem \ref{ThmSharpPPPhik}.

\section{Self-improving via $K$-functionals}\label{Section:SelfImp}

We are interested in the following basic question: Can one obtain 
  an abstract approach to proving   normed inequalities of type
\begin{equation}\label{21}
	\|f\|_B \leq C \, \|f\|_{A_1},
\end{equation}
which are self-improving (in some sense)?
A possible answer relies on the classical interpolation theory. Recall that the interpolation space $(A_0, A_1)_{\theta, p},$  $\theta \in (0, 1),$ $p \in [1, \infty)$, is equipped with the following norm
\begin{equation}\label{CRM}
	\|f\|_{(A_0, A_1)_{\theta, p}}  := \bigg(\int_0^\infty [t^{-\theta} K(t, f; A_0, A_1)]^p \, \frac{dt}{t} \bigg)^{\frac{1}{p}}.
\end{equation}
Applying interpolation to  \eqref{21}, one has
\begin{equation}\label{22}
 \|f\|_{(B, B)_{\theta, p}} \leq C^\theta \, \|f\|_{(B, A_1)_{\theta, p}} \qquad \text{for all} \qquad \theta \in (0, 1).
\end{equation}
Since it is known that  (see e.g. \cite[p. 46]{BerghLofstrom})
\begin{equation}\label{22new}
	 \|f\|_{(B, B)_{\theta, p}} = \frac{1}{c_{\theta, p}} \, \|f\|_B
\end{equation}
with $c_{\theta, p} = ( (1-\theta) \theta p)^{1/p}$, 
 we arrive at the following estimate
\begin{equation}\label{25}
	\|f\|_B \leq C^\theta c_{\theta, p}  \, \|f\|_{(B, A_1)_{\theta, p}} \qquad \text{for all} \qquad \theta \in (0, 1).
\end{equation}
In particular, taking limits as $\theta \to 1^{-}$ in \eqref{25} and invoking  Milman's extrapolation formula\footnote{Under  mild assumptions on the pair $(A_0, A_1)$; see Remark \ref{RemarkInt}.} \cite[Theorem 1]{Milman} (see also \eqref{MEx})
\begin{equation}\label{Mil}
		\lim_{\theta \to 1^-} \, c_{\theta, p} \,  \|f\|_{(A_0, A_1)_{\theta, p}} = \|f\|_{A_1},
\end{equation}
we recover the original inequality  \eqref{21} from \eqref{25}. Specializing \eqref{25} with $(B, A_1) = (L^p, \dot{W}^{1, p})$ (and using Lemma \ref{lemmaAk}), we  obtain the classical BBM estimate \eqref{Intro110}. Hence \eqref{25} can be seen as an extension of \eqref{Intro110}  to an interpolation setting.

The above methodology does not work for the Poincar\'e--Ponce inequalities \eqref{PonceIneq}, because
 the weights in  the right-hand side prevent the use of interpolation technique. 
    However, we are able to overcome this obstruction by  putting \eqref{25} in a more general perspective of $K$-functionals. Specifically, note that \eqref{22new} follows from the trivial formula 
\begin{equation}\label{KFunctTB}
	K(t, f; B, B)  = \min \{1, t\} \, \|f\|_{B}, \qquad \forall\,t > 0,
\end{equation}
and (cf. \eqref{CRM})
$$
	\bigg(\int_0^\infty [t^{-\theta}  \min \{1, t\}]^p \, \frac{dt}{t} \bigg)^{\frac{1}{p}} = \frac{1}{c_{\theta, p}}.
$$
Hence, \eqref{KFunctTB} may be viewed as a sharp version of \eqref{22new}, which is independent of interpolation weights. This discussion served as a motivation to establish the self-improvement result stated in
Theorem \ref{ThmIntPoincare}. Now, we give its proof.

\begin{proof}[Proof of Theorem \ref{ThmIntPoincare}]
The linearity of $T$ and \eqref{Linearity} imply
$$
	K(t, T f;  B, B) \leq \|T\|_0 \, K\bigg( \frac{\|T\|_1}{\|T\|_0} \, t, f; A_0, A_1 \bigg)
$$
for $t > 0$, or equivalently (cf. \eqref{KFunctTB})
$$
	\min\{1, t\} \, \|Tf\|_B \leq \|T\|_0 \, K\bigg( \frac{\|T\|_1}{\|T\|_0} \, t, f; A_0, A_1 \bigg).
$$
As a byproduct, we obtain (after a change of variables)
$$
	\bigg(\int_0^{\frac{\|T\|_1}{\|T\|_0}}  \rho_\varepsilon(t) \, dt \bigg) \,  \|Tf\|^p_{B} \leq \|T\|_1^p \, \int_0^{\frac{\|T\|_1}{\|T\|_0}} \bigg[\frac{K (t, f; A_0, A_1)}{t}\bigg]^p \rho_\varepsilon(t) \, dt.
$$
By \eqref{AssRho1} (taking into account that $\text{supp } \rho_\varepsilon \subset \big(0, \frac{\|T\|_1}{\|T\|_0} \big)$), the previous estimate reads
$$
	\|T f\|_{B}^p  \leq  \|T\|_1^{p} \, \int_0^{\frac{\|T\|_1}{\|T\|_0}} \bigg(\frac{K(t, f; A_0, A_1)}{ t}\bigg)^p \rho_\varepsilon (t)  \, dt.
$$
\end{proof}


Next we justify our assertion that Theorem  \ref{ThmIntPoincare} is a self-improving result. Namely,
   we  show that the right-hand side of \eqref{PonceIntIneq} converges to $\|f\|_{A_1}$ as $\varepsilon \to 0^+$.
 This is a 
 consequence of the more general statement given below.

\begin{lem}\label{LemmaReal}
Let $p\in(0,\infty)$,
$\{\rho_\varepsilon\}_{\varepsilon\in(0,\infty)}$
be a family of functions satisfying \eqref{AssRho1}
and \eqref{AssRho}, and let
$\{\psi_\varepsilon\}_{\varepsilon\in(0,\infty)}$ be
a family of functions satisfying \eqref{AssPsi1}
and \eqref{AssPsi}.
\begin{enumerate}
  \item[{\rm(i)}] Let $g$ be a non-negative function on $(0, \infty)$ such that
  	\begin{equation}\label{LimAs2}
		L := \lim_{t \to 0^+} \, \frac{g(t)}{t} \in [0, \infty],
	\end{equation}
 and if $L < \infty$ we assume further that
	\begin{equation}\label{LimAs1}
	\sup_{t > 0} \, \frac{g(t)}{t} < \infty.
	\end{equation}
	Then
	\begin{equation*}
		\lim_{\varepsilon \to 0^+} \, \bigg[\int_0^\infty \bigg(\frac{g(t)}{t} \bigg)^p \rho_\varepsilon(t) \, dt \bigg]^{\frac{1}{p}} = L.
	\end{equation*}
  \item[{\rm(ii)}] Let $h$ be a non-negative function on
  $(0, \infty)$ such that
  	\begin{equation}\label{LimAspsi2}
	\lim_{t \to \infty} h(t) \in [0, \infty],
	\end{equation}
	and if $\lim_{t \to \infty} h(t) < \infty$ we assume further that
	\begin{equation}\label{LimAspsi1}
	\sup_{t > 0} h(t) < \infty.
	\end{equation}
	Then
	\begin{equation*}
		\lim_{\varepsilon \to 0^+} \,
\bigg[\int_0^\infty \left[h(t)\right]
^p \psi_\varepsilon(t) \, dt \bigg]^{\frac{1}{p}} =
\lim\limits_{t\to\infty}h(t).
	\end{equation*}
\end{enumerate}
\end{lem}

\begin{rem}\label{RemarkInt}
	Let
$\{\rho_\varepsilon\}_{\varepsilon\in(0,\infty)}$
be a family of functions satisfying \eqref{AssRho1}
and \eqref{AssRho} with $\text{supp } \rho_\varepsilon \subset \big(0, \frac{\|T\|_1}{\|T\|_0} \big)$. Applying Lemma \ref{LemmaReal}(i) with $g(t) = K(t, f; A_0, A_1)$,  we derive
\begin{equation}\label{MilGen}
	\lim_{\varepsilon \to 0^+} \left(\int_0^{\frac{\|T\|_1}{\|T\|_0}} \bigg(\frac{K(t, f; A_0, A_1)}{t}\bigg)^p \rho_\varepsilon (t ) \, dt \right)^{\frac{1}{p}} = \lim_{t \to 0^+} \, \frac{K(t, f; A_0, A_1)}{t},
\end{equation}
where the right-hand side defines the norm on the so-called \emph{Gagliardo closure} of $A_1$, which can be identified with $A_1$
(
specifically, when $(A_0, A_1)$ is \emph{Gagliardo closed}, see \cite[Chapter 5]{BennettSharpley} for further details). 
\end{rem}

\begin{proof}[Proof of Lemma \ref{LemmaReal}]
We only prove (i),
 the proof of (ii) is similar.
Assume first that $L < \infty$.
Accordingly, given any $\eta > 0$ there exists $\delta > 0$ such that
$$
	\bigg|\bigg( \frac{g(t)}{t} \bigg)^p -L^p \bigg| \leq \eta \qquad \text{for all} \quad t \in (0, \delta).
$$
Therefore,
\begin{equation}\label{230}
	(L^p-\eta) \int_0^\delta \rho_\varepsilon(t) \, dt \leq \int_0^\delta \bigg( \frac{g(t)}{t} \bigg)^p \rho_\varepsilon(t) \, dt \leq (L^p+\eta) \int_0^\delta \rho_\varepsilon(t) \, dt.
\end{equation}
We now let $\varepsilon \to 0^+$ and invoke \eqref{AssRho} (note that $\delta$ is independent of $\varepsilon$)  and then let  $\eta \to 0^+$. We obtain
\begin{equation}\label{231new}
	\lim_{\varepsilon \to 0^+} \,  \int_0^\delta \bigg( \frac{g(t)}{t} \bigg)^p \rho_\varepsilon(t) \, dt = L^p.
\end{equation}

On the other hand, by \eqref{LimAs1} and \eqref{AssRho1},
\begin{equation*}
	\int_\delta^\infty \bigg(\frac{g(t)}{t} \bigg)^p \rho_\varepsilon(t) \, dt \leq \int_\delta^\infty \rho_\varepsilon(t) \, dt \, \sup_{t > 0} \, \bigg(\frac{g(t)}{t} \bigg)^p = \bigg(1 - \int_0^\delta \rho_\varepsilon(t) \, dt \bigg) \, \sup_{t > 0} \, \bigg(\frac{g(t)}{t} \bigg)^p.
\end{equation*}
In light of \eqref{AssRho}, we have
\begin{equation}\label{231}
	\lim_{\varepsilon \to 0^+} \, \int_\delta^\infty \bigg(\frac{g(t)}{t} \bigg)^p \rho_\varepsilon(t) \, dt =0.
\end{equation}

Combining \eqref{231new} and \eqref{231}, we arrive at
$$
	\lim_{\varepsilon \to 0^+} \,  \int_0^\infty \bigg( \frac{g(t)}{t} \bigg)^p \rho_\varepsilon(t) \, dt = L^p.
$$

The proof of the case $L=\infty$ is even easier: for any $M > 0$ there exists $\delta \in (0, 1)$ such that
$$
	\bigg(\frac{g(t)}{t} \bigg)^p \geq M \qquad \text{for all} \quad t \in (0, \delta),
$$
which implies
$$
\int_0^\infty \bigg(\frac{g(t)}{t} \bigg)^p \rho_\varepsilon(t) \, dt \geq	\int_0^\delta  \bigg(\frac{g(t)}{t} \bigg)^p \rho_\varepsilon(t) \, dt \geq M \int_0^\delta \rho_\varepsilon(t) \, dt.
$$
Hence, by \eqref{AssRho},
$$
\lim_{\varepsilon \to 0^+} \, \int_0^\infty \bigg(\frac{g(t)}{t} \bigg)^p \rho_\varepsilon(t) \, dt 	\geq M,
$$
and thus
$$
	\lim_{\varepsilon \to 0^+}\,	\int_0^\infty \bigg(\frac{g(t)}{t} \bigg)^p \rho_\varepsilon(t) \, dt  = \infty,
$$
completing the proof. 
\end{proof}

\section{Poincar\'e--Ponce inequalities}\label{SectionPP}

\subsection{Basic definitions}
Let $k\in\mathbb{N}$ and let $X$ be a r.i. space on $\R^N$, cf. \cite{BennettSharpley}. The \emph{Sobolev-type space} $\dot{W}^kX$ is defined to
be the set of all the distributions $f$ on $\R^N$ such that $\partial^{\alpha}f\in X$,
for any $\alpha\in\mathbb{Z}_+^N$ with $|\alpha|=k$.
 Moreover, for any $f\in\dot{W}^kX$,
let
\begin{align*}
\|f\|_{\dot{W}^kX}:=\sum_{|\alpha|= k}
\left\|\partial^{\alpha}f\right\|_{X}.
\end{align*}
Note that
$$
	\|f\|_{\dot{W}^kX} \approx \left\|\nabla^k f \right\|_X,
$$
where $\nabla^k f$ denotes the vector of all $k$-th order weak derivatives of $f$ and $\|\nabla^k f \|_X$ is a simplified notation to mean $ \| \, |\nabla^k f| \, \|_X$.
The \emph{(inhomogeneous) Sobolev-type} space $W^kX$
is defined to be the set of all the locally integrable
functions $f$ on $\R^N$ such that $\partial^{\alpha}f\in X$, for any $\alpha\in\mathbb{Z}_+^N$
with $|\alpha|\le k$. Moreover,
for any $f\in W^kX$, let
\begin{align*}
\|f\|_{W^kX}:= \|f\|_X + \|f\|_{\dot{W}^kX}.
\end{align*}

Let $U$ be a measurable bounded set in $\R^N$. By $X(U)$ we mean the local version of $X$ equipped with the norm
\begin{equation}\label{LocalX}
\|f\|_{X(U)} := \|f \mathbf{1}_U\|_X.
\end{equation}

The \emph{$k$-th order moduli of smoothness} of $f \in X(Q)$ is defined by
\begin{equation}\label{DefMod}
	\omega_k (f, t)_{X(Q)} := \sup_{|h| < t} \|\Delta_h^k f\|_{X(Q(k, h))}, \qquad t > 0,
\end{equation}
where the  difference $\Delta^k_h$ is given by
\begin{equation}\label{DefDifference}
\Delta^k_h f (x) := \sum_{j=0}^k (-1)^{j+k} \binom{k}{j} f(x+ j h) \qquad \text{for} \qquad x \in Q(k, h).
\end{equation}
When $k=1$, we denote $\Delta^1_h$ simply by
$\Delta_h$. Moreover, the \emph{$k$-th order moduli of smoothness}
$\omega_k(f,t)_X$ of $f \in X$ is defined as in \eqref{DefMod}
via replacing $X(Q(k,h))$ by $X$.

\subsection{Proof of Theorem \ref{ThmSharpPPPhik}}\label{Subsection:StaDisc}

If $T$ denotes the linear operator which maps $f$ to $[f]$, the equivalence class of $f$ with respect to $\mathcal{P}_{k-1}$, then \eqref{ThmSharpPPPhike} can be rewritten as
$$
 T: \dot{W}^{k} X(Q) \to  X(Q)/\mathcal{P}_{k-1}
$$
with related norm dominated by $C \ell(Q)^k$.  Here, $X^p(Q)/\mathcal{P}_{k-1}$ is equipped with (cf. \eqref{localerror})
$$
	\|f\|_{X(Q)/\mathcal{P}_{k-1}} := E_{k}(f, Q)_X.
$$
On the other hand, we obviously have
\begin{equation}\label{ThmSharpPPPhike1}
\left\|T f\right\|_{X(Q)/\mathcal{P}_{k-1}}
\le \left\|f\right\|_{X(Q)}.
\end{equation}

Applying
Theorem \ref{ThmIntPoincare} with $A_0=X(Q)$,
$A_1=\dot{W}^{k} X(Q)$, and $B=X(Q)/\mathcal{P}_{k-1}$ (and so $\|T\|_0 = 1, \, \|T\|_1 = C \ell(Q)^k$)
\begin{equation}\label{411}
E_{k}(f, Q)_X^p
=\left\|Tf\right\|_{X(Q)/\mathcal{P}_{k-1}}^p \lesssim \ell(Q)^{k p}  \int_{0}^{\ell(Q)^k}
\left(\frac{K(t,f;X(Q),\dot{W}^{k} X (Q))}{t}
\right)^p\rho_{\varepsilon}(t)\,dt.
\end{equation}
Furthermore, using the estimate for the $K$-functional given in Lemma \ref{lemmaAk} (adequately adapted to cubes) and Fubini's theorem, we derive
\begin{align}\label{4.10}
	 \ell(Q)^{k p}  \int_{0}^{\ell(Q)^k}
\left(\frac{K(t,f;X(Q),\dot{W}^{k} X (Q))}{t}
\right)^p\rho_{\varepsilon}(t)\,dt &\\
&\hspace{-5cm} \approx\ell(Q)^{k p}  \int_{0}^{\ell(Q)^k}\frac{1}{t^{\frac{N}{k}+p}}\int_{|h| \le t^{\frac{1}{k}}} \|\Delta^k_hf \|_{X(Q(k, h))}^p \,dh \, \rho_{\varepsilon}(t) \,dt \nonumber \\
& \hspace{-5cm}  \approx \ell(Q)^{kp} \,  \int_{|h| \leq \ell(Q)}  \|\Delta^k_hf \|_{X(Q(k, h))}^p \, \phi_{k, \varepsilon} (|h|)  \, dh. \nonumber
\end{align}
Combining \eqref{411} and \eqref{4.10}, we achieve the desired result. \qed

\subsection{Special cases of Theorem \ref{ThmSharpPPPhik} and comparison with previous  results.}\label{SectionCompare} Next we apply Theorem \ref{ThmSharpPPPhik} for the special choices of $\rho_\varepsilon$ given by Examples \ref{Ex1} (cf. Appendix \ref{SectionAA}) and compare the outcomes with the corresponding Ponce's inequalities obtained in \cite[Section 2]{Ponce}. We stress that all  equivalence constants from this section are independent of $f,$ the cube $Q$, and the smoothness parameter $\varepsilon$ (or $s$).

\begin{cor}
Let $X$ be a r.i. space on $\R^N$, $p\in(0,\infty)$,
and $k\in\mathbb{N}$. Then, for any $f\in X$ and
$s\in[\frac{1}{2},k)$,
\begin{equation}\label{SharpPPX}
E_{k}(f,Q)_X^p
\lesssim (k-s)\ell(Q)^{sp}
\int_{|h| \le \ell(Q)}
\frac{\|\Delta^k_hf\|_{X(Q(k,h))}^p}
{|h|^{sp+N}}\,dh.
\end{equation}
In particular, if $X = L^p(\R^N), \, p \in [1, \infty)$, then
	\begin{equation}\label{SharpPP}
	 	E_{k}(f, Q)_p^p	 \lesssim (k-s) \, \ell(Q)^{s p} \, \int_Q \int_{\widetilde{Q}(k, x)} \frac{|\Delta^k_h f (x)|^p}{|h|^{s p + N}}  \, dh \, dx,
	\end{equation}
	which coincides with the classical BBM estimate \eqref{Intro110} if $k=1$.
\end{cor}

\begin{proof}
Letting $\rho_{\varepsilon}(t)=\varepsilon\ell(Q)^{-k\varepsilon}
t^{\varepsilon-1}{\bf 1}_{(0,\ell(Q)^{k})}(t)$ (cf. Example \ref{Ex1}(i)) in Theorem \ref{ThmSharpPPPhik}, we have $\phi_\varepsilon(t) \lesssim \varepsilon \ell(Q)^{-\varepsilon k} t^{-k p - N + \varepsilon k}$ if $0 < \varepsilon < \varepsilon_0$ and thus \eqref{SharpPPX} follows from \eqref{SharpPPk} (using the change of variables $\varepsilon \leftrightarrow \frac{(k-s)p}{k}$).
\end{proof}

\begin{remark}
In fact, in the above proof one can make use of the equivalence
$$\phi_\varepsilon(t) \approx \varepsilon \ell(Q)^{-\varepsilon k} \left(t^{-k p - N + \varepsilon k}-\ell(Q)^{-k p - N + \varepsilon k}\right).$$
However, the additional negative term does not allow  to sharpen \eqref{SharpPPX}
 because the corresponding integral
   satisfies\footnote{Use that $\Delta^k_h$ annihilates polynomials in $\mathcal{P}_{k-1}$.}
	$$
		\varepsilon \ell(Q)^{-N} \int_{|h| \leq \ell(Q)} \|\Delta^k_h f\|^p_{X(Q(k, h))} \, dh \lesssim \varepsilon  E_k(f, Q)_X^p.
	$$
	
\end{remark}

		As a consequence of Theorem \ref{ThmSharpPPPhik} for 
  $\rho_\varepsilon(t) = \frac{\alpha  t^{\alpha -1}}{\varepsilon^\alpha \ell(Q)^{k \alpha}}  \, \mathbf{1}_{(0, \varepsilon \ell(Q)^k)}(t)$ (cf. Example \ref{Ex1}(ii) with \eqref{27new}, appropriately adapted to the cube $Q$), we establish the following family of estimates.

\begin{cor}\label{CorPlog}
Let $X$ be a r.i. space on $\R^N$, $p, \alpha \in(0,\infty)$,
and $k\in\mathbb{N}$.
\begin{enumerate}[\upshape(i)]
\item Let  $\alpha > p + \frac{N}{k}$.
Then, for any $f\in X$
and $\varepsilon\in(0,1)$,
\begin{equation*}
E_k(f,Q)_X^p \lesssim \varepsilon^{-p-\frac{N}{k}} \ell(Q)^{-k (p+ \frac{N}{k})}
\int_{|h| <\varepsilon^{1/k}\ell(Q)}
\left\|\Delta^k_hf\right\|_{X(Q(k, h))}^p \bigg[ 1 - \bigg(\frac{\varepsilon^{1/k} \ell(Q)}{|h|} \bigg)^{k (p-  \alpha + \frac{N}{k}) } \bigg]     \,dh.
\end{equation*}
\item  Let  $\alpha = p + \frac{N}{k}$.
Then, for any $f\in X$
and $\varepsilon\in(0,1)$,
\begin{equation*}
	E_k(f,Q)_X^p \lesssim \varepsilon^{-\alpha} \ell(Q)^{-k \alpha} \int_{|h| < \varepsilon^{1/k} \ell(Q)} \log \bigg(\frac{\varepsilon^{1/k} \ell(Q)}{|h|} \bigg) \left\|\Delta^k_h f\right\|_{X(Q(k, h))}^p \, dh.
\end{equation*}
\item Let  $\alpha < p + \frac{N}{k}$.
Then, for any $f\in X$
and $\varepsilon\in(0,1)$,
\begin{equation*}
E_k(f,Q)_X^p\lesssim \varepsilon^{-p-\frac{N}{k}} \ell(Q)^{-k (p+ \frac{N}{k})}
\int_{|h| <\varepsilon^{1/k}\ell(Q)}
\left\|\Delta^k_hf\right\|_{X(Q(k, h))}^p  \bigg[ \bigg(\frac{\varepsilon^{1/k} \ell(Q)}{|h|} \bigg)^{k(p-\alpha + \frac{N}{k})}- 1 \bigg]  \,dh.
\end{equation*}
\end{enumerate}
\end{cor}

\begin{remark}
	Working with $X = L^p(\R^N), \, p \in [1, \infty)$, and $k=1$, Ponce \cite[Corollary 2.4]{Ponce}  obtained, for the unit cube $Q$,
	\begin{equation}\label{PR1}
		\int_Q |f-f_Q|^p \lesssim \varepsilon^{-p-N} \underset{|x-y| \leq \varepsilon}{\int_Q \int_Q}  |f(x)-f(y)|^p \, dx \, dy.
	\end{equation}
	According to Corollary \ref{CorPlog}(i), the right-hand side of \eqref{PR1} can be reinforced to say that
	\begin{equation}\label{431new}
	\int_Q |f-f_Q|^p \lesssim  \varepsilon^{-p-N} \underset{|x-y| \leq \varepsilon}{\int_Q \int_Q}  |f(x)-f(y)|^p  \bigg[1 - \bigg(\frac{\varepsilon}{|x-y|} \bigg)^{p -\alpha + N} \bigg] \, dx \, dy
	\end{equation}
	for any $\alpha > p + N$. Here, the hidden equivalence constant is independent of $\varepsilon$, but it clearly depends on $\alpha$ (otherwise, let
   $\alpha \to (p+N)^{+}$). Note that the additional term $1 - (\frac{\varepsilon}{|x-y|})^{p -\alpha + N}$ in \eqref{431new} is monotone with respect to $\alpha$.

	On the other hand, if $\alpha < p+N$ then an application of Corollary \ref{CorPlog}(iii) gives
	\begin{equation*}
	\int_Q |f-f_Q|^p \lesssim \varepsilon^{-p-N} \, \underset{|x-y| \leq \varepsilon}{\int_Q \int_Q} |f(x)-f(y)|^p  \bigg[ \bigg(\frac{   \varepsilon }{|x-y|} \bigg)^{p-\alpha+N} -1\bigg] \, dx \, dy.
	\end{equation*}
\end{remark}

\begin{cor}
Let $X$ be a r.i. space on $\R^N$, $p\in(0,\infty)$,
and $k\in\mathbb{N}$. Then, for any $f\in X$ and
$\varepsilon\in(0,1)$,
	\begin{align}\nonumber
E_k(f,Q)_X^p
&\lesssim \frac{\ell(Q)^{k p}}{|\log \varepsilon|}
\int_{\ell(Q)\varepsilon^{1/k}<|h| \le \ell(Q)}
\frac{\|\Delta^k_hf\|^p_{X(Q(k,h))}}{|h|^{kp+N}}\,dh \nonumber \\
&\quad+ \frac{\ell(Q)^{-N}}{|\log \varepsilon|}\left[ \frac{1}{\varepsilon^{p+\frac{N}{k}}}
\int_{|h|\le\ell(Q)\varepsilon^{1/k}}\left\|\Delta^k_hf\right\|^p_{X(Q(k,h))}\,dh -
\int_{|h|\le\ell(Q)}\left\|\Delta^k_hf\right\|^p_{X(Q(k,h))}\,dh\right]. \label{PonImpLog}
	\end{align}
\end{cor}

\begin{proof}
	We will apply Theorem \ref{ThmSharpPPPhik} with
$\rho_\varepsilon(t)=\frac{1}{|\log \varepsilon|}
t^{-1}{\bf 1}_{(\ell(Q)^k\varepsilon,\ell(Q)^k)}(t)$ (cf. Example \ref{Ex1}(iii)).
In this case,
the function $\phi_\varepsilon$ is given by
\begin{equation}\label{PhiSpecialLog}
	\phi_\varepsilon(t) = \left\{\begin{array}{cl}  \frac{\ell(Q)^{-k p - N}}{|\log \varepsilon| (k p + N)} \, (\varepsilon^{-p - \frac{N}{k}} -1), & \text{if} \quad t \in (0, \ell(Q) \varepsilon^{\frac{1}{k}}), \\
		\frac{1}{|\log \varepsilon| (k p + N)} \, (t^{-k p - N} - \ell(Q)^{-k p - N}),  & \text{if} \quad t \in (\ell(Q) \varepsilon^{\frac{1}{k}}, \ell(Q)).
		       \end{array}
                        \right.
\end{equation}
Thus direct computations lead to
\begin{align}
\int_{|h| \leq \ell(Q)}  \left\|\Delta^k_h f \right\|_{X(Q(k, h))}^p
\phi_{\varepsilon} (|h|) \,dh & \approx \frac{1}{|\log \varepsilon|}
\int_{\ell(Q)\varepsilon^{1/k}<|h| \le \ell(Q)}
\frac{\|\Delta^k_hf\|^p_{X(Q(k,h))}}{|h|^{kp+N}}\,dh \nonumber \\
&\hspace{-4.5cm}\quad+ \frac{\ell(Q)^{-kp-N}}{|\log \varepsilon|}\left[ \frac{1}{\varepsilon^{p+\frac{N}{k}}}
\int_{|h| \le\ell(Q)\varepsilon^{1/k}}\left\|\Delta^k_hf\right\|^p_{X(Q(k,h))}\,dh -
\int_{|h| \le\ell(Q)}\left\|\Delta^k_hf\right\|^p_{X(Q(k,h))}\,dh\right], \label{423new}
\end{align}
and the desired estimate \eqref{PonImpLog} follows from \eqref{SharpPPk}.
\end{proof}

\begin{remark}
	Assume, for simplicity, that $Q$ is the unit cube. Specializing \eqref{PonImpLog} with $X = L^p (\R^N), \, p \in [1, \infty),$ and $k=1$, we derive
		\begin{align}
\int_Q |f-f_Q|^p
&\lesssim \frac{1}{|\log \varepsilon|} \, \underset{|x-y| > \varepsilon}{\int_Q \int_Q}
\frac{|f(x)-f(y)|^p}{|x-y|^{p+N}}\,dx \, dy \nonumber \\
&\quad+ \frac{1}{|\log \varepsilon|}\left[ \frac{1}{\varepsilon^{p+N}} \underset{|x-y| \leq \varepsilon }{\int_Q \int_Q}  |f(x)-f(y)|^p \,dx \, dy -
\int_Q \int_Q |f(x)-f(y)|^p \,dx \, dy\right]. \label{425new}
	\end{align}
	{Here we take into account  two special cases of \eqref{BBM} (cf. \cite[formulas (46) and (48)]{Brezis} for further details):
	$$
		\lim_{\varepsilon \to 0^+} \frac{1}{\varepsilon^{p+N}} \underset{|x-y| \leq \varepsilon }{\int_Q \int_Q}  |f(x)-f(y)|^p \,dx \, dy = \frac{C_{N, p}}{N+p} \, \|\nabla f\|^p_{L^p(Q)}
	$$
	and
		\begin{equation*}
		\lim_{\varepsilon \to 0^+} \frac{1}{|\log \varepsilon|} \, \underset{|x-y| > \varepsilon}{\int_Q \int_Q} \frac{|f(x)-f(y)|^p}{|x-y|^{p+N}}\,dx \, dy = C_{N, p} \, \|\nabla f\|^p_{L^p(Q)};
	\end{equation*}
	see also Corollaries \ref{CorMSX2} and  \ref{Coro633} below for more general versions of these formulas. In particular,  \eqref{425new} complements the following inequality due to Ponce \cite[Corollary 2.3]{Ponce}:
	$$
		\int_Q |f-f_Q|^p \lesssim \frac{1}{|\log \varepsilon|} \, \underset{|x-y| > \varepsilon}{\int_Q \int_Q}
\frac{|f(x)-f(y)|^p}{|x-y|^{p+N}}\,dx \, dy.
	$$
	}

\end{remark}

\section{Quantitative Gaussian Sobolev inequalities}

\subsection{Proof of Theorem \ref{ThmGSL}}\label{ProofG}

Let $W_\mathcal{L} L^p(\R^N, \gamma_N)$ be the $L^p$-Sobolev space associated with $\mathcal{L}$, i.e.,
$$
	\|f\|_{W_\mathcal{L} L^p(\R^N, \gamma_N)} := \|\mathcal{L} f\|_{L^p(\R^N, \gamma_N)}.
$$
Accordingly,  inequality \eqref{GaussSobolev} can be viewed in terms of boundedness properties of the linear operator $T f = f-m(f)$, specifically,
\begin{equation}\label{GS1Proof}
	T: W_\mathcal{L} L^p(\R^N, \gamma_N) \rightarrow L^p (\log L)^p (\R^N, \gamma_N).
\end{equation}
On the other hand, clearly, there holds 
\begin{equation}\label{GSProof2}
	T: L^p(\R^N, \gamma_N) \rightarrow L^p(\R^N, \gamma_N).
\end{equation}
Note that the norm of $T$ relative to both \eqref{GS1Proof} and \eqref{GSProof2} is independent of $N$.
It follows from  \eqref{GS1Proof} and \eqref{GSProof2} that, with a constant independent of $N$,
\begin{equation*}
	K(t, Tf; L^p(\R^N, \gamma_N), L^p (\log L)^p (\R^N, \gamma_N)) \lesssim K(t, f; L^p(\R^N, \gamma_N), W_\mathcal{L} L^p(\R^N, \gamma_N) )
\end{equation*}
and thus a simple change of variables leads to
\begin{align}
	\int_0^{1} \bigg(\frac{K((1-\log t)^{-1}, T f; L^p(\R^N, \gamma_N), L^p (\log L)^p (\R^N, \gamma_N)) }{(1-\log t)^{-1}} \bigg)^p \rho_\varepsilon( (1-\log t)^{-1}) \, \frac{dt}{t (1-\log t)^2} & \nonumber\\
	&\hspace{-11cm} = \int_0^{1} \bigg(\frac{K(t, T f; L^p(\R^N, \gamma_N), L^p (\log L)^p (\R^N, \gamma_N)) }{t} \bigg)^p \, \rho_\varepsilon(t) \, dt\nonumber \\
	& \hspace{-11cm} \lesssim  \int_0^1 \bigg(\frac{K(t , f; L^p(\R^N, \gamma_N), W_\mathcal{L} L^p(\R^N, \gamma_N) ) }{t} \bigg)^p  \, \rho_\varepsilon(t) \, dt. \label{65G}
\end{align}

Recall the well-known fact that $W_\mathcal{L} L^p(\R^N, \gamma)$ is the domain space related to the infinitesimal  generator of the analytic semigroup $\{G_t\}_{t > 0}$  on $L^p(\R^N, \gamma_N)$. In particular, the Ditzian--Ivanov formula for $K$-functionals (see \cite[Theorem 5.1]{DitzianIvanov}) can be applied to get
\begin{equation}\label{DIForm}
	K(t, f; L^p(\R^N, \gamma_N), W_\mathcal{L} L^p(\R^N, \gamma_N) ) \approx   \|f-G_t f \|_{L^p(\R^N, \gamma_N)}.
\end{equation}
More precisely, an inspection of the proof of the estimate $\lesssim$ in \eqref{DIForm} given in \cite{DitzianIvanov} shows that the corresponding constant only depends on the following quantities
$$
	C_1 = \sup_{t > 0} \|G_t\|_{L^p(\R^N, \gamma_N) \to L^p(\R^N, \gamma_N)} \qquad \text{and} \qquad C_2 =\sup_{t > 0} t \, \|\mathcal{L} G_t\|_{L^p(\R^N, \gamma_N) \to L^p(\R^N, \gamma_N)}.
$$
Note that both $C_1$ and $C_2$ are independent of $N$. More precisely, it is well known that $C_1 \leq 1$ and, on the other hand,  $C_2 \leq C(p) C(\frac{p}{p-1})$, where $C(p) = \big((2 \pi)^{-1/2} \int_\R |s|^p e^{-s^2/2} \, ds \big)^{1/p}$; we refer to  \cite[p. 687, Proposition 4.1]{Kosov} for further details.

It follows from \eqref{DIForm} and  Fubini's theorem that
\begin{equation}\label{GSProof4}
	\int_0^1 \bigg(\frac{K(t , f; L^p(\R^N, \gamma_N), W_\mathcal{L} L^p(\R^N, \gamma_N) ) }{t} \bigg)^p  \, \rho_\varepsilon(t) \, dt \lesssim   \int_{\R^N} \int_0^1 \frac{|f(x)-G_t f(x)|^p}{t^p}
 \, \rho_\varepsilon(t) \, dt \, d\gamma_N(x).
\end{equation}

Next we turn our attention to the $K$-functional related to $(L^p, L^p (\log L)^p)$ on $(\R^N, \gamma_N)$. At this stage, additional obstacles arise since $L^p (\log L)^p$ can not be described as a classical interpolation space of $L^p$. However, we are able to overcome such obstacles by applying the modern theory of \emph{limiting interpolation}, cf. \cite{DominguezTikhonov}  and the references therein. Indeed,  we have that $L^p (\log L)^p$ can be characterized as a limiting interpolation space of $L^p$, i.e.,
\begin{equation}\label{GSProof5}
	L^p (\log L)^p (\R^N, \gamma_N) = (L^p(\R^N, \gamma_N), L^\infty(\R^N, \gamma_N))_{0, p, \frac{1}{p'}},
\end{equation}
where
\begin{equation}\label{LimIntNewForm}
	\|f\|^p_{(L^p(\R^N, \gamma_N), L^\infty(\R^N, \gamma_N))_{0, p, \frac{1}{p'}}} := \int_0^1 [(1-\log t)^{1/p'} K(t, f; L^p(\R^N, \gamma_N), L^\infty(\R^N, \gamma_N))]^p \, \frac{dt}{t}.
\end{equation}
The  introduction of meaningful limiting interpolation spaces requires further modifications than just letting  $\theta = 0, 1$ in the classical norm  $\|\cdot\|_{(A_0, A_1)_{\theta, p}}$ (cf. \eqref{CRM}). Indeed, in the ordered case $A_1 \subset A_0$,   only the behavior of the $K$-functional near the origin plays a role (i.e., the integral in \eqref{LimIntNewForm} is defined on $(0, 1)$ rather than on $(0, \infty)$) and additional logarithmic weights are needed.
In fact, \eqref{GSProof5} easily follows from
$$
	K(t, f; L^p(\R^N, \gamma_N), L^\infty(\R^N, \gamma_N)) \approx \bigg(\int_0^{t^p} (f^{*}_{\gamma_N} (\lambda))^p \, d\lambda\bigg)^{1/p}
$$
with constants independent of $N$, see \cite[Theorem 5.2.1, p. 109]{BerghLofstrom}. Therefore one can make use of known Holmstedt's formulas for limiting interpolation (cf. \cite[Lemma 2.1]{DominguezTikhonov}) in combination with \eqref{GSProof5}, to establish, for any $t \in (0, 1)$,
\begin{align*}
	&K( (1-\log t)^{-1}, f; L^p(\R^N, \gamma_N), L^p (\log L)^p (\R^N, \gamma_N)) \\
& \quad\approx  K( (1-\log t)^{-1}, f; L^p(\R^N, \gamma_N), (L^p(\R^N, \gamma_N), L^\infty(\R^N, \gamma_N))_{0, p, \frac{1}{p'}}) \\
	& \quad \approx (1-\log t)^{-1} \, \bigg(\int_t^1 (1-\log u)^{p/p'} K(u, f; L^p(\R^N, \gamma_N), L^\infty(\R^N, \gamma_N))^p \, \frac{du}{u} \bigg)^{1/p} + \|f\|_{L^p(\R^N, \gamma_N)}   \\
	& \quad \approx (1-\log t)^{-1} \, \bigg((1-\log t)^p \, \int_0^{t} (f^*_{\gamma_N}(\lambda))^p \, d \lambda  + \int_{t}^1 (f^*_{\gamma_N} (\lambda))^p (1-\log \lambda)^p \, d \lambda  \bigg)^{1/p} + \|f\|_{L^p(\R^N, \gamma_N)}.
\end{align*}
Furthermore,  all  equivalence constants in the above formulas are independent of $N$. Hence
\begin{align}
	\int_0^{1} \bigg(\frac{K((1-\log t)^{-1}, T f; L^p(\R^N, \gamma_N), L^p (\log L)^p (\R^N, \gamma_N)) }{(1-\log t)^{-1}} \bigg)^p \rho_\varepsilon( (1-\log t)^{-1}) \, \frac{dt}{t (1-\log t)^2} & \nonumber\\
	&\hspace{-11cm} \gtrsim \int_0^{1} (1-\log t)^p \, \int_0^{t} ((Tf)^*_{\gamma_N}(\lambda))^p \, d \lambda \,  \rho_\varepsilon( (1-\log t)^{-1}) \, \frac{dt}{t (1-\log t)^2}  \nonumber  \\
& \hspace{-10cm}  + \int_0^{1} \int_{t}^{1} ((T f)^*_{\gamma_N} (\lambda))^p (1-\log \lambda)^p \, d \lambda \,  \rho_\varepsilon( (1-\log t)^{-1}) \, \frac{dt}{t (1-\log t)^2} \nonumber\\
& \hspace{-11cm}  \approx \int_0^{1} ((T f)^*_{\gamma_N}(t))^p \, \Upsilon_\varepsilon (t) \, dt, \label{68G}
\end{align}
where the last estimate follows from Fubini's theorem and \eqref{UpsilonDef}.

Putting together \eqref{65G}, \eqref{GSProof4} and \eqref{68G}, we arrive at
$$
\int_0^{1} ((T f)^*_{\gamma_N}(t))^p \, \Upsilon_\varepsilon (t) \, dt \lesssim    \int_{\R^N} \int_0^1 \frac{|f(x)-G_t f(x)|^p}{t^p}
 \, \rho_\varepsilon(t) \, dt \, d\gamma_N(x).
$$
\qed

\subsection{Theorem \ref{ThmGSL} sharpens classical Gaussian Sobolev inequalities}\label{Section42}

We claim that the Gaussian Sobolev inequality \eqref{GaussSobolev} follows from \eqref{ThmGSL1eps} as $\varepsilon \to 0^+$.

\begin{thm}\label{Theorem49}
Let $\{\rho_\varepsilon\}_{\varepsilon > 0}$ be a family of functions with $\emph{supp } \rho_\varepsilon \subset (0, 1)$ satisfying \eqref{AssRho1}  and \eqref{AssRho}. Let $\{\Upsilon_\varepsilon\}_{\varepsilon > 0}$ be the related functions \eqref{UpsilonDef}. Then
	\begin{equation}\label{611G}
		\lim_{\varepsilon \to 0^+} \,	\int_0^{1}  [f-m(f)]^{* p}_{\gamma_N}(t) \, \Upsilon_\varepsilon(t) \, dt =   \|f\|_{L^p (\log L)^p(\R^N, \gamma_N)}^p
	\end{equation}
	and
	\begin{equation}\label{612G}
	\lim_{\varepsilon \to 0^+}\, 	 \int_{\R^N} \int_0^1 \frac{|f(x)-G_t f(x)|^p}{t^p}
 \, \rho_\varepsilon(t) \, dt \, d\gamma_N(x) =  \|\mathcal{L} u\|_{L^p(\R^N, \gamma_N)}^p.
	\end{equation}

\end{thm}

\begin{proof}

We start by showing \eqref{611G}. In view of \eqref{UpsilonDef}, we can express
\begin{equation*}
	\int_0^{1}  [f-m(f)]^{* p}_{\gamma_N}(t) \, \Upsilon_\varepsilon(t) \, dt   = \int_0^{1}  \bigg(\frac{g(t)}{t} \bigg)^p \rho_\varepsilon(t) \, dt,
\end{equation*}
where
$$
	\bigg(\frac{g(t)}{t} \bigg)^p = \int_{e^{1-t^{-1}}}^{1} (1-\log u)^p [f-m(f)]^{* p}_{\gamma_N}(u) \, du  + \frac{1}{t^p}  \int_0^{e^{1-t^{-1}}} [f-m(f)]^{* p}_{\gamma_N}(u) \, du.
$$
An application of Lemma \ref{LemmaReal}(i) gives
\begin{align*}
	\lim_{\varepsilon \to 0^+} \int_0^{1}  [f-m(f)]^{* p}_{\gamma_N}(t) \, \Upsilon_\varepsilon(t) \, dt  &=   \int_0^{1} (1-\log u)^p [f-m(f)]^{* p}_{\gamma_N}(u) \, du  \\
	& \hspace{1cm} + \lim_{t \to 0^+} \frac{1}{t^p}  \int_0^{e^{1-t^{-1}}} [f-m(f)]^{* p}_{\gamma_N}(u) \, du.
	\end{align*}
Therefore the desired result \eqref{611G} will be shown provided that
	$$
		\lim_{t \to 0^+} \frac{1}{t^p}  \int_0^{e^{1-t^{-1}}} [f-m(f)]^{* p}_{\gamma_N}(u) \, du = 0.
	$$
	Indeed, assume without loss of generality that $f$ is bounded (i.e., $f^*(0) < \infty$), then
	\begin{align*}
		\lim_{t \to 0^+} \frac{1}{t^p}  \int_0^{e^{1-t^{-1}}} [f-m(f)]^{* p}_{\gamma_N}(u) \, du & = \lim_{t \to 0^+}  \frac{e^{1-t^{-1}} [f-m(f)]^{* p}_{\gamma_N}(e^{1-t^{-1}}) }{p t^{p+1}}   =0.
	\end{align*}

	The assertion \eqref{612G} follows as an application of Theorem \ref{ThmSemGen}(i) below for the special choice $X = L^p(\R^N, \gamma_N), T_t = G_t$ and $\alpha =1$.
\end{proof}

\subsection{Examples}\label{Section43}
Next we write down some special cases of Theorem \ref{ThmGSL} related  to different  $\{\rho_\varepsilon\}_{\varepsilon > 0}$ given in Examples \ref{Ex1}, Appendix \ref{SectionAA}. In particular, taking $\rho_\varepsilon(t) = \varepsilon t^{\varepsilon-1} \mathbf{1}_{(0, 1)}(t)$ and $\rho_\varepsilon (t) = \varepsilon t^{p-\varepsilon-1} \mathbf{1}_{(0, 1)}(t)$, we are able to establish the analogue of the BBM estimate  \eqref{Intro110} in the Gaussian setting.

\begin{cor}\label{CorGauss1}
Let $p \in (1, \infty)$. Then there exists a positive constant $C$, which is independent of $N$ and $0 < \varepsilon <  1$,  such that
\begin{equation}\label{CorGaussSob1}
	\int_0^{1} \{ (1-\log t)^{\varepsilon} [f-m(f)]^*_{\gamma_N} (t)\}^p \, dt \leq C   \varepsilon (1- \varepsilon) \, \int_{\R^N} \int_0^1 \frac{|f(x)-G_t f(x)|^p}{t^{ \varepsilon p}} \, \frac{dt}{t} \, d \gamma_N(x)
	\end{equation}
	 for every $f \in L^p(\R^N, \gamma_N)$. In particular
	 \begin{equation}\label{CorGaussSob1new2}
	\|f-m(f) \|_{L^p(\R^N, \gamma_N)}^p \leq C   \varepsilon (1- \varepsilon) \, \int_{\R^N} \int_0^1 \frac{|f(x)-G_t f(x)|^p}{t^{ \varepsilon p}} \, \frac{dt}{t} \, d \gamma_N(x).
	\end{equation}
\end{cor}

 In particular, Corollary \ref{CorGauss1} gives a positive answer to a question posed by Kosov in \cite{Kosov}. Specifically, the author  showed  in \cite[Theorem 3.2]{Kosov} the following logarithmic Sobolev inequality for Gaussian--Besov spaces:
 letting $p \in [1, \infty)$ and $\alpha \in (0, 1)$,  there exists a constant $C$, which is independent of $N$, such that\footnote{The left-hand side of \eqref{Kosov} in the original statement in \cite{Kosov}
 was equivalently written in terms of
$
	\int_{\R^N} |f|^p (\log_+ |f|)^{\beta p/2} \, d \gamma_N;
$ cf. \cite[Lemma 6.12, p. 252]{BennettSharpley}.
}
\begin{equation}\label{Kosov}
	\int_0^1 (1-\log t)^{\beta p/2} (f_{\gamma_N}^*(t))^p \, dt \lesssim \|f\|_{L^p(\R^N, \gamma_N)}^p +  \int_0^\infty t^{-\alpha p} \sigma(f, t)_p^p \, \frac{dt}{t}, \qquad \beta \in (0, \alpha),
\end{equation}
where $\sigma(f,t)_p$ is the special modulus of smoothness of a function 
  $f \in L^p(\R^N, \gamma_N)$ (see \cite[p. 664]{Kosov} for the precise   definition).
  A natural important problem, explicitly posed in \cite[Remark 3.2]{Kosov}, is to sharpen \eqref{Kosov} for
   the limiting case $\beta = \alpha$.
   Answering this question,
    inequality \eqref{CorGaussSob1} does not only refine \eqref{Kosov} 
    in terms of parameters, but also provides  the corresponding BBM-type behavior  
    with respect to the smoothness parameter $\varepsilon$.
    Indeed, using
 (cf. \cite[Lemma 3.1 and p. 675]{Kosov})
$$
	\|f-G_t f\|_{L^p(\R^N, \gamma_N)} \lesssim \sigma(f, t^{1/2})_p,
$$
it follows from \eqref{CorGaussSob1} that
$$
	\int_0^{1} \{ (1-\log t)^{\alpha/2} [f-m(f)]^*_{\gamma_N} (t)\}^p \, dt \lesssim (2- \alpha) \, \int_0^1 t^{- \alpha p} \sigma(f, t)_p^p  \, \frac{dt}{t}
$$
for $\alpha \in (0, 2)$ and $p \in (1, \infty)$.

We close this section with further examples of Theorem \ref{ThmGSL}.

\begin{cor}
	Let $\alpha > 0$ and $p \in (1, \infty)$. Then there exists a positive constant $C$, which is independent of $N$ and $0 < \varepsilon < \varepsilon_0 < 1$,  such that
	$$
	 \int_{\varepsilon}^1 (1-\log t)^p  [f-m(f)]^{* p}_{\gamma_N}(t) \, dt	 \leq C \,	(1-\log \varepsilon)^\alpha \int_{\R^N} \int_0^{(1-\log \varepsilon)^{-1}} \frac{|f(x)-G_t f(x)|^p}{t^{p-\alpha}} \, \frac{dt}{t} \, d \gamma_N(x)
	$$
	and
\begin{align*}
&\int_{\varepsilon}^1 (1-\log t)^p (1 - \log ((1-\log \varepsilon)^{-1} (1-\log t))) \, [f-m(f)]^{* p}_{\gamma_N}(t) \, dt\\
&\quad \leq C \,      \int_{\R^N} \int_{(1-\log \varepsilon)^{-1}}^1 \frac{|f(x) -G_t f(x)|^p}{t^p} \, \frac{dt}{t} \, d \gamma_N(x)
\end{align*}	
for every $f \in L^p(\R^N, \gamma_N)$.
\end{cor}

\section{John--Nirenberg inequalities sharpened}

We will use the following notation. Given two normed spaces $X$ and $Y$, by $X \hookrightarrow Y$ we mean that $X \subset Y$ and the corresponding embedding is continuous.
By
 $f^{\#}_Q$ we define the local version of \eqref{SMax}, that is,
 $f^{\#}_Q(x) :=  \sup_{\substack{Q' \ni x \\ Q' \subset Q} } \,  \fint_{Q'} |f-f_{Q'}|.$

\subsection{Proof of Theorem \ref{ThmJNSharp}}\label{Section51}

Assume first that $|Q|=1$. Since
$$
	L^\infty(\R^N) \hookrightarrow L^p(Q) \qquad \text{and} \qquad L^p(\R^N) \hookrightarrow L^p(Q)
$$
with related norms $1$, we can apply Theorem \ref{ThmIntPoincare} to obtain
\begin{equation}\label{JNS23}
	\|f\|_{L^p(Q)} \leq  \left(\int_0^1 \bigg(\frac{K(t, f; L^p(\R^N), L^\infty(\R^N))}{t}\bigg)^p \rho_\varepsilon(t) \, dt \right)^{\frac{1}{p}}.
\end{equation}
Recall that the $K$-functional related to $(L^p, L^\infty)$ can be characterized as  (see e.g. \cite[Theorem 5.2.1]{BerghLofstrom})
\begin{equation*}
	K(t, f; L^p(\R^N), L^\infty(\R^N)) \approx  \bigg(\int_0^{t^p} (f^*(u))^p \, du \bigg)^{1/p},
\end{equation*}
where the equivalence constants are independent of $p$ and   $N$. As a consequence, we can estimate the right-hand side of \eqref{JNS23} as follows
 \begin{align}\label{EquivLinfty}
	  &\left(\int_0^1 \bigg(\frac{K(t, f; L^p(\R^N), L^\infty(\R^N))}{t}\bigg)^p \rho_\varepsilon(t) \, dt \right)^{\frac{1}{p}} \\ &\quad\approx \bigg(\int_0^1 \frac{1}{t^p} \int_0^{t^p} (f^*(u))^p \, du \, \rho_\varepsilon(t) \, dt  \bigg)^{\frac{1}{p}}
=  \bigg(\int_0^1 (f^*(t))^p \, \eta_{\varepsilon, p}(t) \, dt \bigg)^{\frac{1}{p}} \nonumber
\end{align}
and inserting this into \eqref{JNS23}, we obtain
$$
	\|f\|_{L^p(Q)} \leq C_N \, \bigg(\int_0^1 (f^*(t))^p \, \eta_{\varepsilon, p}(t) \, dt \bigg)^{\frac{1}{p}}.
$$
In particular, the previous inequality applied to $f^{\#}$ gives
\begin{equation}\label{JNS24}
	\|f^{\#}\|_{L^p(Q)} \leq C_N \, \bigg(\int_0^1 (f^{\#*}(t))^p \, \eta_{\varepsilon, p}(t) \, dt \bigg)^{\frac{1}{p}}.
\end{equation}
On the other hand, the classical Fefferman--Stein inequality asserts
\begin{equation}\label{JNS25}
	\|f-f_Q\|_{L^p(Q)} \leq C_N  p \, \|f^{\#}_Q\|_{L^p(Q)}.
\end{equation}
 As a combination of \eqref{JNS24} and \eqref{JNS25}, we arrive at
$$
	\|f-f_Q\|_{L^p(Q)} \leq C_N p \, \bigg(\int_0^1 (f^{\#*}(t))^p \, \eta_{\varepsilon, p}(t) \, dt \bigg)^{\frac{1}{p}}.
$$
This  completes the proof of Theorem \ref{ThmJNSharp} if $|Q|=1$.
The general case of $Q$ as stated in \eqref{ThmJNSharp2} follows from the previous estimate by standard scaling arguments.
\qed

\subsection{Proof of Theorem \ref{ThmExtrJN}}\label{Section52}
By monotonicity properties of rearrangements and \eqref{BMODef}, we obtain
\begin{align*}
\frac{1}{|Q|} \, \int_0^{|Q|} (f^{\#*}(t))^p \, \eta_{\varepsilon, p} \Big(\frac{t}{|Q|} \Big) \, dt & \leq f^{\#*}(0)^p \, \frac{1}{|Q|} \, \int_0^{|Q|}  \eta_{\varepsilon, p} \Big(\frac{t}{|Q|} \Big) \, dt \\
& = \|f\|_{\text{BMO}(\R^N)}^p \, \int_0^1 \eta_{\varepsilon, p} (t) \, dt \\
& = \|f\|_{\text{BMO}(\R^N)}^p \, \int_0^1  \int_{t^{\frac{1}{p}}}^1 \frac{\rho_\varepsilon(u) }{u^p} \, du \, dt \\
& = \|f\|_{\text{BMO}(\R^N)}^p \, \int_0^1 \rho_\varepsilon(u) \, du \\
& = \|f\|_{\text{BMO}(\R^N)}^p,
\end{align*}
where we have used \eqref{AssRho1} in the last step. This proves \eqref{ThmExtrJN1}.

Next we focus on \eqref{ThmExtrJN2} with $|Q|=1$ (the general case follows from standard scaling arguments). Taking into account \eqref{ThmJNSharp1} and changing the order of integration, we have
\begin{equation*}
	 \int_0^1 (f^*(t))^p \, \eta_{\varepsilon, p}(t) \, dt  = \int_0^1 \frac{ g(t)^p}{t^p} \, \rho_\varepsilon(t) \, dt,
\end{equation*}
where
$$
	g(t) := \bigg( \int_0^{t^p} (f^*(u))^p \, du \bigg)^{1/p}.
$$
Assume $f \in L^\infty(\R^N)$. Then
$$
	g(t) \leq t  f^*(0) = t \, \|f\|_{L^\infty(\R^N)}, \qquad \forall t > 0,
$$
and
$$
	\lim_{t \to 0^+} \bigg(\frac{g(t)}{t} \bigg)^p = \lim_{t \to 0^+} \frac{\int_0^t (f^*(u))^p \, du}{t} = \lim_{t \to 0^+} (f^*(t))^p = \|f\|_{L^\infty(\R^N)}^p.
$$
Therefore we are in a position to apply Lemma \ref{LemmaReal}(i), obtaining
$$
	\lim_{\varepsilon \to 0^+} \bigg( \int_0^1 (f^*(t))^p \, \eta_{\varepsilon, p}(t) \, dt \bigg)^{1/p}  = \|f\|_{L^\infty(\R^N)}.
$$
For  $f^{\#}$ we then derive that
$$
\lim_{\varepsilon \to 0^+} \bigg( \int_0^1 (f^{\#*}(t))^p \, \eta_{\varepsilon, p}(t) \, dt \bigg)^{1/p}  = \|f\|_{\text{BMO}(\R^N)}.
$$
\qed

\subsection{Examples}\label{Section53}
 We write down some special cases of Theorem \ref{ThmJNSharp} according to the list of Examples \ref{Ex1}, Appendix \ref{SectionAA}.

\begin{cor}
	Let $f \in L^p(\R^N), \, p \in [1, \infty)$. Then
	$$
		\bigg(\fint_{Q} |f-f_{Q}|^p \bigg)^{\frac{1}{p}} \leq C_N  p \varepsilon^{\frac{1}{p}}   \,\bigg(\frac{1}{|Q|^\varepsilon} \int_0^{|Q|}t^{\varepsilon}  f^{\#*}(t)^p \, \frac{dt}{t} \bigg)^{\frac{1}{p}}
	$$
	for every $0 < \varepsilon \leq \varepsilon_0 < 1$ and $Q$.
\end{cor}

\begin{cor}\label{c7.16}
Let $f\in L^p(\R^N), \, p\in[1,\infty)$, and $\alpha > 0$.
\begin{enumerate}
  \item[{\rm(i)}] If $\alpha \ne p$, then	
  \begin{equation*}
 \fint_{Q}|f-f_Q|^p
  \le C_N^p  \frac{p^p \alpha}{ \varepsilon |Q| (\alpha-p)} \left[
  \int_{0}^{\varepsilon|Q|}f^{\# *}(t)^p\,dt
  -\frac{1}{(\varepsilon |Q|)^{\frac{\alpha}{p}-1}}
  \int_{0}^{\varepsilon|Q|}t^{\frac{\alpha}{p}}
  f^{\# *}(t)^p\,\frac{dt}{t}\right]
  \end{equation*}
  for every
  $ \varepsilon  \in (0, 1)$ and $Q$.
  \item[{\rm(ii)}] If $\alpha = p$, then
  \begin{align*}
  \fint_Q|f-f_Q|^p
  \le C_N^p
  \frac{p^p}{\varepsilon|Q|}\int_{0}^{\varepsilon|Q|}
  \log\left(\frac{\varepsilon|Q|}{t}\right)
  f^{\# *}(t)^p\,dt
  \end{align*}
  for every
  $\varepsilon \in (0, 1)$ and $Q$.
\end{enumerate}
\end{cor}

%
%
%

\begin{cor}
Let $f\in L^p(\R^N), \, p\in [1,\infty)$. Then
\begin{align*}
\fint_Q|f-f_Q|^p
  \le C_N^p \frac{p^p}{|\log\varepsilon||Q|}
  \int_{0}^{|Q|}\left(\max\left\{
  \frac{t}{|Q|},\varepsilon\right\}^{-1}-1\right)f^{\# *}(t)^p
  \,dt
\end{align*}
for every
$0 < \varepsilon \leq \varepsilon_0 < 1$ and $Q$.
\end{cor}


\section{Limiting formulas}

\subsection{New Bourgain--Brezis--Mironescu  and Maz'ya--Shaposhnikova formulas in the classical setting}\label{SubSection6.1}
In this section, we investigate
the Bourgain--Brezis--Mironescu and Maz'ya--Shaposhnilova formulas in Sobolev spaces.


\begin{thm}\label{ThmGenSobk}
Let $X$ be a  r.i. space on $\R^N$ with absolutely continuous norm,
$p\in(0,\infty)$,
and $k\in\mathbb{N}$. Let $\{\rho_\varepsilon\}_{\varepsilon > 0}$
be a family of functions satisfying \eqref{AssRho1}. Define (cf. \eqref{gsfg})
\begin{equation*}
	\phi_\varepsilon (t) := \int_{t^k}^\infty u^{-p-\frac{N}{k}} \rho_\varepsilon(u) \, du, \qquad
\forall\, t > 0.
\end{equation*}
Then, for any  $f\in\dot{W}^kX$ and $\varepsilon\in(0,\infty)$,
\begin{equation}\label{ThmGenSobk1.2}
\int_{\R^N} \left\|\Delta_h^k f\right\|_X^p \,
 \phi_\varepsilon(|h|) \,dh
 \lesssim  \left\|\nabla^k f\right\|_X^p.
\end{equation}
If, in addition, $\{\rho_\varepsilon\}_{\varepsilon > 0}$ satisfies \eqref{AssRho} then\footnote{As usual, $x^\alpha = x_1^{\alpha_1} \cdots x_N^{\alpha_N}$ for $x = (x_1, \ldots, x_N) \in \R^N$ and $\alpha =( \alpha_1, \ldots, \alpha_N) \in \Z^N_+$.}
\begin{align}\label{ThmGenSobk1.1}
\lim\limits_{\varepsilon\to0^+}
\int_{\R^N}\left\|\Delta^k_hf\right\|_{X}^p
\phi_\varepsilon(|h|)\,dh
=\frac{1}{kp+N}\int_{\mathbb{S}^{N-1}}
\bigg\|\sum_{|\alpha|=k}
\omega^\alpha\partial^\alpha f\bigg\|_{X}^p\,d\sigma^{N-1}(\omega)
\end{align}
for any $f\in\dot{W}^kX$, and
	\begin{equation}\label{ThmGenSobk1}
		\lim_{\varepsilon \to 0^+}
\int_{\R^N} \left\|\Delta_h^k f\right\|_X^p
\, \phi_\varepsilon(|h|) \, dh  \approx
\left\|\nabla^k f\right\|_X^p
	\end{equation}
for any $f\in W^kX$.
\end{thm}

\begin{proof}
Let $f\in\dot{W}^kX$.
By \cite[Proposition 1.4.5]{g14}, for almost every $x\in\R^N$
and for any $h \in\R^N$,
\begin{align}\label{dhk0}
\Delta^k_hf(x)&=
\int_{[0,1]^k}\sum_{|\alpha|=k}
h^{\alpha}\partial^\alpha f\left(x+(s_1+\cdots +s_k)h\right)
\,ds_1\cdots\,ds_k
\end{align}
(cf. \eqref{DefDifference}),  which combined with the translation invariance
of $X$ implies that
\begin{align}\label{dhk}
\left\|\Delta^k_hf\right\|_X
\lesssim|h|^k\left\|\nabla^k f\right\|_X.
\end{align}
Using this estimate
together with the Fubini theorem and the assumption
\eqref{AssRho1} on $\rho_{\varepsilon}$, we obtain
\begin{align*}
\int_{\R^N}\left\|\Delta^k_hf\right\|_X^p
\phi_\varepsilon(|h|)\,dh & \lesssim\left\|\nabla^k f\right\|_X^p
\int_{\R^N}|h|^{kp}\phi_\varepsilon(|h|)\,dh\\
& \approx\left\|\nabla^k f\right\|_X^p
\int_{0}^{\infty}r^{kp+N-1}\phi_\varepsilon(r)\,dr\\
&=\left\|\nabla^k f\right\|_X^p
\int_{0}^{\infty}u^{-p-\frac{N}{k}}\rho_{\varepsilon}(u)\,
\int_{0}^{u^{\frac{1}{k}}}r^{kp+N-1}\,dr \, du\\
&\approx \left\|\nabla^k f\right\|_X^p
\int_{0}^{\infty}\rho_{\varepsilon}(u)\,du =\left\|\nabla^k f\right\|_X^p,
\end{align*}
which completes the proof of \eqref{ThmGenSobk1.2}.

Next, we show \eqref{ThmGenSobk1.1}.
To do this, for any $t\in(0,\infty)$,
let
\begin{equation}\label{DefG}
g(t) := \bigg(\frac{1}{t^{\frac{N}{k}}}
\, \int_{|h| \leq t^{\frac{1}{k}}}
\left\|\Delta^k_h f\right\|_X^p \, d h \bigg)^{\frac{1}{p}}.
\end{equation}
Then
\begin{equation*}
\int_{\R^N}\left\|\Delta^k_hf\right\|_X^p
\phi_\varepsilon(|h|)\,dh= \int_{0}^{\infty}\left[\frac{g(t)}
{t}\right]^p\rho_{\varepsilon}(t)\,dt.
\end{equation*}
In addition, applying \eqref{dhk}, we find that, for any $t\in(0,\infty)$,
$g(t)\lesssim t\|\nabla^kf\|_X$, that is, \eqref{LimAs1} holds.
Therefore, by Lemma \ref{LemmaReal},
to prove \eqref{ThmGenSobk1.1} it is enough to show
\begin{align}\label{ThmGenSobke2}
\lim\limits_{t\to0^+}
\left[\frac{g(t)}{t}\right]^p=
\frac{1}{kp+N}\int_{\mathbb{S}^{N-1}}
\bigg\|\sum_{|\alpha|=k}
\omega^\alpha\partial^{\alpha}f\bigg\|_X^p\,d\sigma^{N-1}(\omega).
\end{align}
We first assume $k=1$ in \eqref{ThmGenSobke2}. Note that the set
$\{f\in C^\infty(\R^N):\ \nabla f\in C_{{\rm c}}^\infty(\R^N,\R^N)\}$
is dense in $\dot{W}^1X$ (see \cite[Proposition 2.15]{DLYYZ23}). This implies that
we can further assume that
$f\in C^\infty(\R^N)$ and $\nabla f$ is a
Lipschitz function compactly supported in $\Omega$.
Then, by \eqref{dhk0}, we have,
for any $x \in \R^N$, $\omega \in \mathbb{S}^{N-1}$,
and $t$ small enough,
\begin{equation*}
\left|\frac{\Delta_{t\omega}f(x)}{t}
-\nabla f(x)\cdot\omega\right|  =\left|\int_{0}^{1}
\left[\nabla f(x+st\omega)-\nabla f(x)\right]\cdot
\omega\,ds\right|\lesssim t{\bf 1}_{2\Omega}(x).
\end{equation*}
Using this and the lattice property of $X$,
we find that
	$$
	\left| \frac{\|\Delta_{t \omega} f\|_X}{t} -
\left\|\nabla f \cdot \omega\right\|_X \right|
\leq	\bigg\| \frac{\Delta_{t \omega} f}{t}
- \nabla f \cdot \omega
\bigg\|_X \lesssim t \, \| \mathbf{1}_{2 \Omega}\|_X\to0 \qquad \text{as} \qquad  t\to0^+.
	$$
Combining this, \eqref{dhk}, the assumption
$|\nabla f|\in X$, and the Lebesgue dominated convergence theorem,
we further conclude that
\begin{align*}
&\lim\limits_{t\to0^+}
\int_{\mathbb{S}^{N-1}}\left(\frac{\|\Delta_{t\omega}f\|_X}{t}\right)^p
\,d\sigma^{N-1}(\omega)\\
&\quad=\int_{\mathbb{S}^{N-1}}\lim\limits_{t\to0^+}
\left(\frac{\|\Delta_{t\omega}f\|_X}{t}\right)^p
\,d\sigma^{N-1}(\omega)
=\int_{\mathbb{S}^{N-1}}\left\|\nabla f\cdot\omega\right\|_X^P
\,d\sigma^{N-1}(\omega).
\end{align*}
This, together with 
  the L'H\^{o}pital rule, further implies that
\begin{align}\label{ThmGenSobke3}
\lim\limits_{t\to0^+}
\left[\frac{g(t)}{t}\right]^p
&=\lim\limits_{t\to0^+}\frac{\int_{0}^{t}\int_{\mathbb{S}^{N-1}}
\|\Delta_{\lambda \omega}f\|_X^p\,d\sigma^{N-1}(\omega) \, \lambda^{N-1}\,d\lambda}
{t^{p+N}} \\
&=\frac{1}{p+N}\lim\limits_{t\to0^+}
\int_{\mathbb{S}^{N-1}}\left(\frac{\|\Delta_{t\omega}f\|}{t}\right)^p
\,d\sigma^{N-1}(\omega)\notag\\
&=\frac{1}{p+N}\int_{\mathbb{S}^{N-1}}\left\|\nabla f\cdot\omega\right\|_X^P
\,d\sigma^{N-1}(\omega).\notag
\end{align}
Thus \eqref{ThmGenSobke2} with $k=1$ follows.
For $k > 1$,
observe that, for any $h\in\R^N$,
\begin{align*}
\sum_{|\alpha|=k}
h^\alpha\partial^{\alpha}f
=\sum_{|\beta|=k-1}
\nabla\left(h^{\beta}\partial^\beta f\right)\cdot h.
\end{align*}
Then an argument similar to the above with $f$ replaced
by $\{\partial^{\beta}f\}_
{|\beta|=k-1}\subset\dot{W}^1X$ yields \eqref{ThmGenSobke2} for
any $k > 1$.

Finally, we show \eqref{ThmGenSobk1}. Indeed,
by a change of variables, Lemma \ref{lemmaAk},
and the Tonelli theorem,
we find that
\begin{align}
\int_{0}^{\infty}
\left[\frac{K(t,f;X,\dot{W}^kX)}{t}\right]^p
\rho_\varepsilon(t)\,dt
&\approx\int_{0}^{\infty}
\frac{1}{t^N}\int_{|h|\le t}
\left\|\Delta_h^kf\right\|_X^p \,
\rho_{\varepsilon}(t^k)t^{k(1-p)-1}\, dh \,dt \nonumber\\
&\approx\int_{\R^N}\left\|\Delta_h^kf\right\|_X^p
\int_{|h|^k}^{\infty}\rho_{\varepsilon}(t)
t^{-p-\frac{N}{k}}\,dt\,dh
=\int_{\R^N}\left\|\Delta_h^kf\right\|_X^p
\phi_\varepsilon(|h|)\,dh.\label{ThmGenSobke1}
\end{align}
In addition, similarly to \cite[(1.10)]{KolomoitsevTikhonov},
we have, for any $f\in X$,
\begin{align*}
\sup_{t\in(0,\infty)}\frac{\omega_k(f,t)_X}{t^k}
\approx\left\|\nabla^k f\right\|_X.
\end{align*}
From this, \eqref{ThmGenSobke1}, and Lemma \ref{LemmaReal}(i),
we deduce that, for any $f\in W^kX$,
\begin{align*}
\lim\limits_{\varepsilon\to0^+}\int_{\R^N}
\left\|\Delta^k_hf\right\|_X^p\phi_\varepsilon(|h|)\,dh
&\approx\sup_{t\in(0,\infty)}
\left[\frac{K(t,f;X,\dot{W}^kX)}{t}\right]^p\\
&\approx\sup_{t\in(0,\infty)}
\left[\frac{\omega_k(f,t)_X}{t^k}\right]^p
\approx\left\|\nabla^kf\right\|_X^p.
\end{align*}
This finishes the proof of \eqref{ThmGenSobk1} and hence Theorem \ref{ThmGenSobk}.
\end{proof}

The Maz'ya--Shaposhnikova counterpart of Theorem \ref{ThmGenSobk} reads as follows.

\begin{thm}\label{ThmGenSobMSk}
Let $X$ be a  r.i. space on $\R^N$ with absolutely continuous norm, $p \in(0,\infty)$,
and $k\in\mathbb{N}$.
Let $\{\psi_\varepsilon\}_{\varepsilon > 0}$
be a family of functions satisfying
\eqref{AssPsi1} and \eqref{AssPsi}.
Define
\begin{equation}\label{612new}
	\varphi_{\varepsilon} (t) := \int_{t^k}^\infty u^{-\frac{N}{k}}
\psi_\varepsilon(u) \, du, \qquad \forall\,t > 0.
\end{equation}
Then, for any $f \in X$,
\begin{equation}\label{ThmGenSobMSk1}
\lim\limits_{\varepsilon\to0^+}
\int_{\R^N}\left\|\Delta^k_h f\right\|_X^p
\varphi_{\varepsilon}(|h|)\,dh
=\frac{|\mathbb{S}^{N-1}|}{N}
\lim\limits_{|h|\to\infty}
\left\|\Delta^k_hf\right\|_X^p,
\end{equation}
where $\lim_{|h|\to\infty}
\|\Delta^k_hf\|_X$ exists with $\lim_{|h|\to\infty}
\|\Delta^k_hf\|_X \approx \|f\|_X$.
\end{thm}

\begin{proof}
Let $f\in X$, we first show that
$\lim_{|h|\to\infty}\|\Delta^k_hf\|_X$
exists. To do this, we now assume that $f$ has compact support, i.e.,
$\supp (f)\subset B(0,R)$ for some $R\in(0,\infty)$.
Then, for any $j\in\mathbb{N}$ and $h\in\R^N$,
$\supp(f(\cdot+jh))\subset B(-jh,R)$. Thus $\{\supp(f(\cdot+jh))\}_{j=0}^k$
are pairwise disjoint provided that $|h|>2R$. This further implies that, for any
$h\in\R^N$ with $|h|>2R$ and\footnote{As usual, $d_f$ denotes the distribution function of $f$.} $\lambda\in(0,\infty)$,
\begin{align}
d_{\Delta^k_hf}(\lambda)
&=\left|\left\{x\in\R^N:\ \left|
\sum_{j=0}^{k}(-1)^{k-j}\binom{k}{j}
|f(x+jh)|\right|>\lambda\right\}\right| \nonumber\\
&=\sum_{j=0}^{k}\left|\left\{x\in B(-jh,R):\
\binom{k}{j}|f(x+jh)|>\lambda\right\}\right| =\sum_{j=0}^{k}d_f\left(\binom{k}{j}^{-1}\lambda\right). \label{614}
\end{align}
Thus, for any $h\in\R^N$ with $|h|>2R$ and for any $t\in(0,\infty)$,
$$
	(\Delta^k_h f)^*(t)  = \inf
 \left\{\lambda \in(0,\infty) : \sum_{j=0}^{k}d_f\left(\binom{k}{j}^{-1}
 \lambda\right)\leq t \right\} =: \tilde{g} (t).
$$
In particular, $\tilde{g}$ is a decreasing and right-continuous function. Applying Ryff's theorem \cite[Corollary 7.8, p. 86]{BennettSharpley}, there exists a measurable
function $b$ on $\R^N$ such that $b^* = \tilde{g}$, and so
$$
\|b\|_X= \left\|\Delta^k_h f\right\|_X
$$
for any $h\in\R^N$ with $|h|>2R$.
To prove that $\lim_{|h| \to \infty} \|\Delta^k_h f\|_X$ exists and finite,
it remains to show that $b \in X$. Indeed, let $j_0\in\{0,\ldots,k\}$
be such that $\binom{k}{j_0}=\max_{j\in\{0,\ldots,k\}}\binom{k}{j}$.
Then, for any $t\in(0,\infty)$,
monotonicity properties of $d_f$ yield
\begin{align*}
&\left\{\lambda\in(0,\infty):\
(k+1)d_{f}\left(\binom{k}{j_0}^{-1}\lambda\right)\le t\right\}\\
&\quad\subset\left\{\lambda \in(0,\infty) : \sum_{j=0}^{k}d_f\left(\binom{k}{j}^{-1}
 \lambda\right)\leq t \right\}
 \subset\left\{\lambda\in(0,\infty):\ d_f(\lambda)\le t\right\}.
\end{align*}
Taking the infimum over all $\lambda\in(0,\infty)$, we obtain
$$
	f^*(t) \leq  \tilde{g}(t) = b^*(t) \leq
\binom{k}{j_0}f^*\bigg(\frac{t}{k+1}\bigg),
$$
which implies that
$$
 \|b\|_X \approx \|f\|_X.
$$
Thus, for $f$ compactly supported,
$\lim_{|h|\to\infty}\|\Delta^k_hf\|_X$ exists and
$\lim_{|h|\to\infty}\|\Delta^k_hf\|_X\approx\|f\|_X$. For general $f \in X$, the same result follows from standard density arguments.

%

Next, we prove \eqref{ThmGenSobMSk}. Changing the order of integration, we arrive at
\begin{equation*}
\int_{\R^N}
\left\|\Delta_h^k f\right\|_X^p
\varphi_{\varepsilon}(|h|)\, dh = \int_{0}^{\infty}\left[g(t)\right]^p\psi_{\varepsilon}
(t)\,dt,
\end{equation*}
where $g$ is given by \eqref{DefG}.
Since $\|\Delta^k_hf\|_X\le2^k\|f\|_X$, we have $g(t)\lesssim\|f\|_X$. In order to prove the first formula of \eqref{ThmGenSobMSk1}, by Lemma \ref{LemmaReal}(ii), it is enough to show
\begin{equation}\label{ThmGenSobMSke1}
\lim\limits_{t\to\infty}
\left[g(t)\right]^p
=\frac{|\mathbb{S}^{N-1}|}{N}\lim\limits_{|h|\to\infty}
\left\|\Delta^k_hf\right\|_X^p.
\end{equation}
The latter follows similarly as in  \eqref{ThmGenSobke3}.

Regarding the equivalence in \eqref{ThmGenSobMSk1},
similarly to \eqref{ThmGenSobke1},
we have, for any $f\in X$,
\begin{align*}
\int_{0}^{\infty}\left[K\left(t,f;X,\dot{W}^kX\right)\right]^p
\psi_\varepsilon(t)\,dt &\approx\int_{0}^{\infty}\frac{1}{t^N}
\int_{|h|\le t}\left\|\Delta_h^kf\right\|_{X}^p \,
\psi_\varepsilon(t^k)t^{k-1}\, dh \, dt\notag\\
&\approx \int_{\R^N}
\left\|\Delta_h^k f\right\|_X^p \,
\varphi_{\varepsilon}(|h|)\, dh.\notag
\end{align*}
Applying this, Lemma \ref{LemmaReal}(ii),
\eqref{lemmaAke0}, and \cite[Property 1(d)]{KolomoitsevTikhonov},
we then conclude that, for any $f\in X$,
\begin{equation*}
\lim\limits_{\varepsilon\to0^+}
\int_{\R^N}\left\|\Delta^k_h f\right\|_X^p
\varphi_{\varepsilon}(|h|)\,dh \approx\lim\limits_{t\to\infty} K\left(t,f;X,\dot{W}^kX\right)^p
\approx\lim\limits_{t\in(0,\infty)}
\omega_k(f,t)^p_X\approx\|f\|_X^p.
\end{equation*}
This completes  the proof of \eqref{ThmGenSobMSk1}
and hence Theorem \ref{ThmGenSobMSk}.
\end{proof}

For special   choices of $X$ (in particular, the \emph{Lorentz space} 
  $X = L^{p, q}(\R^N)$), the exact value of
$\lim_{|h|\to\infty}\|\Delta^k_h f\|_X$ in \eqref{ThmGenSobMSk1} can be easily computed.
Recall that $\|f\|_{L^{p, q}(\R^N)} := \big(\int_0^\infty (t^{1/p} f^*(t))^q \, \frac{dt}{t} \big)^{1/q}$ and, in particular, $\|f\|_{L^{p, p}(\R^N)} = \|f\|_{L^p(\R^N)}.$

\begin{exam}\label{3.31}
Let $p, q \in [1,\infty)$ and $k\in\mathbb{N}$. Then,
for any $f\in L^{p, q}(\R^N)$,
\begin{equation}\label{LimLorentz}
\lim\limits_{|h|\to\infty}
\left\|\Delta^k_hf\right\|^q_{L^{p, q}(\R^N)} = p \,  \int_0^\infty t^q \bigg[\sum_{j=0}^k d_f \bigg(\binom{k}{j}^{-1} t \bigg) \bigg]^{\frac{q}{p}} \, \frac{dt}{t}.
\end{equation}
In particular, if $p=q$ then
\begin{equation}\label{618}
\lim\limits_{|h|\to\infty}
\left\|\Delta^k_hf\right\|^p_{L^{p}(\R^N)}
= \sum_{j=0}^{k}
\binom{k}{j}^p
\|f\|_{L^p(\R^N)}^p
\end{equation}
and if $k=1$ then
\begin{equation}\label{LpExDif}
	\lim_{|h| \to \infty} \left\|\Delta_hf\right\|_{L^{p, q}(\R^N)} = 2^{1/p} \, \|f\|_{L^{p, q}(\R^N)}.
\end{equation}
The proof of \eqref{LimLorentz} is an immediate consequence of the well-known fact that
$$
	\|f\|_{L^{p, q}(\R^N)}^q = p  \int_0^\infty t^q d_f(t)^{q/p} \, \frac{dt}{t}
$$
and \eqref{614}.
\end{exam}
%
%
%
%
%

\subsection{Examples}\label{Section6.2}

Next we apply Theorems \ref{ThmGenSobk} and \ref{ThmGenSobMSk} for some special choices of families $\{\rho_\varepsilon\}_{\varepsilon > 0}$ and $\{\psi_\varepsilon\}_{\varepsilon > 0}$. Throughout this subsection, unless otherwise is stated, we assume that $X$ is a r.i. space on $\R^N$ with absolutely continuous norm.

\begin{cor}\label{CorMSX}
	Let $p \in(0,\infty)$
and $k\in\mathbb{N}$. Then, for any $f \in W^k X$,
	\begin{align}\label{CorMSXs1}
		\lim_{\varepsilon \to 1^-} (1-\varepsilon)^{\frac{1}{p}} \,
\bigg(\int_{\R^N} \frac{\|\Delta_h^k f\|_X^p}{|h|^{k\varepsilon p + N}} \, dh  \bigg)^{\frac{1}{p}} =\left[\frac{1}{kp}
\int_{\mathbb{S}^{N-1}}
\left\|\sum_{|\alpha|=k}\omega^\alpha
\partial^\alpha f\right\|_X^p
\,d\sigma^{N-1}(\omega)\right]^{\frac{1}{p}}
 \approx \left\|\nabla^k f\right\|_X
	\end{align}
	and
		\begin{equation}\label{CorMSXs0}
		\lim_{\varepsilon \to 0^+} \varepsilon^{\frac{1}{p}} \, \bigg(\int_{\R^N}
 \frac{\|\Delta_h^k f\|_X^p}{|h|^{k\varepsilon p + N}} \, dh  \bigg)^{\frac{1}{p}}
 =\left(\frac{|\mathbb{S}^{N-1}|}{kp}\right)^{\frac{1}{p}}
 \lim_{|h|\to\infty}\left\|\Delta^k_hf\right\|_X\approx\|f\|_X.
	\end{equation}
\end{cor}

\begin{rem}
	Letting $X = L^p(\R^N), \, p \in [1, \infty)$, in the previous result, we recover the higher order analogues of the classical Bourgain--Brezis--Mironescu and Maz'ya--Shaposhnikova formulas
	\begin{equation}\label{623}
		\lim_{\varepsilon \to k^-} (k-\varepsilon)  \int_{\R^N} \int_{\R^N} \frac{|\Delta_h^k f(x)|^p}{|h|^{\varepsilon p + N}} \, dx\, dh =\frac{1}{p}  \int_{\R^N}
\int_{\mathbb{S}^{N-1}}
\bigg|\sum_{|\alpha|=k}\omega^\alpha
\partial^\alpha f (x)\bigg|^p
\,d\sigma^{N-1}(\omega) \, dx
	\end{equation}
	and (cf. \eqref{618})
	\begin{equation}\label{625sunew}
			\lim_{\varepsilon \to 0^+} \varepsilon  \int_{\R^N} \int_{\R^N} \frac{|\Delta_h^k f(x)|^p}{|h|^{\varepsilon p + N}} \, dx\, dh =\frac{|\mathbb{S}^{N-1}|}{p}   \sum_{j=0}^{k}
\binom{k}{j}^p
\|f\|_{L^p(\R^N)}^p.
	\end{equation}
	Note that  formula \eqref{623}, as well as certain variants of both \eqref{623} and \eqref{625sunew}, have been already derived in \cite[p. 332]{KaradzhovMilmanXiao}. Concerning \eqref{623}, we also refer to \cite{Borghol, BIK, Ferreira}.
\end{rem}

\begin{rem}\label{Remark627}
For the Lorentz space $X = L^{p, q}(\R^N), \, p \in (1, \infty), \, q \in [1, \infty)$ (or $p=q=1$) and $k=1$,  Corollary \ref{CorMSX} gives\footnote{For a function $f = f(x, h)$, the notation $\|f\|_{X_x}$ and $\|f\|_{X_h}$ indicates that  the $X$-norm is computed with respect to the variable $x$ and $h$, respectively.}
		\begin{equation}\label{Lorentz1}
		\lim_{\varepsilon \to 1^-}  \big\|\|Q_\varepsilon\|_{L^{p, q}_x(\R^N)} \big\|_{L^p_h(\R^N)} \approx \|\nabla f\|_{L^{p, q}(\R^N)}
	\end{equation}
	and
	$$
		\lim_{\varepsilon \to 0^+}  \big\|\|Q_\varepsilon\|_{L^{p, q}_x(\R^N)} \big\|_{L^p_h(\R^N)} \approx \|f\|_{L^{p, q}(\R^N)},
	$$
where
	$$
		Q_\varepsilon (x, h) = (\varepsilon (1-\varepsilon))^{\frac{1}{p}} \, \frac{f (x+h)-f(x)}{|h|^{\varepsilon + \frac{N}{p}}}.
	$$
We also recall an analogue of  \eqref{Lorentz1} from the recent paper \cite[Theorem 5.23]{DGPYYZ}
	\begin{equation*}
			\lim_{\varepsilon \to 1^-}  \big\|\|Q_\varepsilon\|_{L^p_h(\R^N) } \big\|_{L^{p, q}_x(\R^N)} \approx \|\nabla f\|_{L^{p, q}(\R^N)}.
	\end{equation*}
	We observe that despite switching the norms under the limit,  the outcome is still the same in both cases.
\end{rem}

\begin{proof}[Proof of Corollary \ref{CorMSX}]

We show \eqref{CorMSXs0}.
To do so,
let $\psi_\varepsilon(t)=\varepsilon p t^{-\varepsilon p-1}
{\bf 1}_{(1,\infty)}(t)$ (cf. Example \ref{Ex1Dual}(i) in Appendix \ref{SectionAA}). Then,
by a simple calculation, we find that
\begin{align*}
\varphi_{\varepsilon}(t)=\frac{\varepsilon p}
{\varepsilon p+\frac{N}{k}}
\left\{{\bf 1}_{(0,1)}(t)+t^{-k\varepsilon p-N}
{\bf 1}_{[1,\infty)}(t)\right\}.
\end{align*}
Therefore, from \eqref{ThmGenSobMSk1},
we infer that, for any $f\in X$,
\begin{equation}\label{CorMSXe1}
\lim\limits_{\varepsilon\to 0^+}
\frac{k\varepsilon p}
{k\varepsilon p+N}\left(
\int_{|h|<1}\left\|\Delta^k_hf\right\|_X^p\,dh
+\int_{|h|\ge 1}\frac{\|\Delta^k_hf\|_X^p}{|h|^{k\varepsilon p+N}}
\,dh\right)=\frac{|\mathbb{S}^{N-1}|}{N}\lim_{|h|\to\infty}
\left\|\Delta^k_hf\right\|_X^p.
\end{equation}
Further, we observe that the first limit on the left-hand side vanishes because $\left\|\Delta^k_hf\right\|_X\lesssim\|f\|_X$, and
$$
	\lim\limits_{\varepsilon\to 0^+}
\frac{k\varepsilon p}
{k\varepsilon p+N} \int_{|h|\le 1}\frac{\|\Delta^k_hf\|_X^p}{|h|^{k\varepsilon p+N}}
\,dh =0 \qquad \text{for any} \qquad f \in \dot{W}^k X,
$$
because $\left\|\Delta^k_hf\right\|_X\lesssim
|h|^k\left\|\nabla^kf\right\|_X$. Then \eqref{CorMSXs0} is shown.

The proof of \eqref{CorMSXs1} follows similar ideas as above by now invoking Theorem \ref{ThmGenSobk} with $\rho_\varepsilon(t)=\varepsilon pt^{\varepsilon p-1}
{\bf 1}_{(0,1)}(t)$ (cf. Example \ref{Ex1}(i), Appendix \ref{SectionAA}). Further details are left to the reader.
\end{proof}

Applying Theorems \ref{ThmGenSobk} and \ref{ThmGenSobMSk}  for the special choices  $\rho_\varepsilon(t) = \frac{\alpha}{\varepsilon} (\frac{t}{\varepsilon})^{\alpha-1} \mathbf{1}_{(0, \varepsilon)}(t), \, \alpha > 0$, and $\psi_{\varepsilon}(t)=
\frac{\alpha}{\varepsilon^\alpha t^{\alpha+1}}{\bf 1}
_{(\varepsilon^{-1},\infty)}(t)$ (see Examples \ref{Ex1}(ii) with \eqref{27new} and \ref{Ex1Dual}(ii)), we arrive at the following result.

\begin{cor}\label{CorMSX2}
	Let $p\in(0,\infty)$ and $k\in\mathbb{N}$.
	\begin{enumerate}[\upshape(i)]
	\item If $\alpha \in (0, p + \frac{N}{k})$ then
	\begin{equation}\label{628new}
		\lim_{\varepsilon \to 0^+}
\frac{1}{\varepsilon^{\frac{k \alpha}{p}}} \,
\left(\int_{|h| < \varepsilon}
\frac{\|\Delta_h^k f\|_X^p}{|h|^{k p + N - k \alpha}}
\, dh  \right)^{\frac{1}{p}} \approx \left\|\nabla^k f\right\|_X,
	\end{equation}
	 if $\alpha = p + \frac{N}{k}$ then
	\begin{equation}\label{628new2}
	\lim_{\varepsilon \to 0^+} \, \frac{1}{\varepsilon^{k+\frac{N}{p}}}
\, \left[
\int_{|h| < \varepsilon} \left\|\Delta_h^k f
\right\|_X^p \,
\log \left(\frac{\varepsilon}{|h|} \right)
 \, dh \right]^{\frac{1}{p}} \approx \left\|\nabla^k f\right\|_X,
\end{equation}
and if $\alpha \in (p + \frac{N}{k}, \infty)$ then
\begin{equation}\label{628new3}
	\lim_{\varepsilon \to 0^+} \, \frac{1}{\varepsilon^{k+\frac{N}{p}}}
\, \left[
\int_{|h| < \varepsilon} \left\|\Delta_h^k f
\right\|_X^p
 \, dh \right]^{\frac{1}{p}} \approx \left\|\nabla^k f\right\|_X.
\end{equation}
	\item If $\alpha > 0$ then
	\begin{equation}\label{629new}
		\lim_{\varepsilon \to \infty} \varepsilon^{\frac{k \alpha}{p}} \,
\left(\int_{\R^N} \frac{\|\Delta_h^k f\|_X^p}
{\max\{|h|, \varepsilon\}^{N + \alpha k}}
\, dh \right)^{\frac{1}{p}} \approx \|f\|_X.
	\end{equation}
	\end{enumerate}
\end{cor}

\begin{proof}
	To show \eqref{628new}, we further take into account that
	$
		|h|^{-k p - N + k \alpha} \lesssim |h|^{-k p - N + k \alpha}- \varepsilon^{-p-\frac{N}{k} +\alpha}
	$
	for $|h| < (\frac{\varepsilon}{2})^{\frac{1}{k}}$. Similar estimates are applied to get \eqref{628new2} and \eqref{628new3}.
\end{proof}

Let $\{\rho_\varepsilon\}_{\varepsilon\in(0,1)}$
and $\{\psi_{\varepsilon}\}_{\varepsilon\in(0,1)}$
be the same as, respectively, in
Examples \ref{Ex1}(iii) and \ref{Ex1Dual}(iii). Applying Theorems \ref{ThmGenSobk} and \ref{ThmGenSobMSk}, we obtain the following borderline result.

\begin{cor}\label{Coro633}
Let $p\in(0,\infty)$
and $k\in\mathbb{N}$. Then
\begin{equation*}
	\lim_{\varepsilon \to 0^+} \, \frac{1}{|\log \varepsilon|^{\frac{1}{p}}}
\, \left(\int_{|h|\ge \varepsilon}
\frac{\|\Delta_h^k f\|_X^p}{|h|^{kp+N}} \, dh \right)
^{\frac{1}{p}} \approx \left\|\nabla^k f\right\|_X
\end{equation*}
and
$$
	\lim_{\varepsilon \to \infty} \,
\frac{1}{|\log \varepsilon|^{\frac{1}{p}}} \,
\left(\int_{ |h| < \varepsilon} \frac{\|\Delta_h^k f\|_X^p}{|h|^N} \,
dh \right)^{\frac{1}{p}} \approx \|f\|_X.
$$
\end{cor}

\subsection{Further Bourgain--Brezis--Mironescu and Maz'ya--Shaposhnikova formulas}\label{Section63}

The next result is   in some sense intermediate between the recent BBM-type characterizations of Sobolev norms given in \cite{DGPYYZ} and \cite[Proposition 2.5]{KolyadaLerner}. Note that our assumption on Minkowski's inequality (see \eqref{Minkow} below) plays the role of the boundedness of the maximal operator on the dual space of $X^{1/p}$, the \emph{$1/p$-convexification}\footnote{Recall that $\|f\|_{X^{1/p}} := \| \, |f|^{1/p} \, \|_X^p$.} of $X$, in \cite{DGPYYZ}.

\begin{theorem}\label{ThmGenSobk'}
Let $p\in(0,\infty)$, $k\in\mathbb{N}$, and let
$X$ be a r.i. space on $\R^N$ with absolutely continuous norm and
satisfying the following Minkowski's integral inequality
	\begin{equation}\label{Minkow}
		\| \,\|f(x, h)\|_{L^p_x} \, \|_{X_h} \lesssim \|\,\|f(x, h)\|_{X_h} \, \|_{L^p_x}
	\end{equation}
	for any measurable function $f$ on $\R^{2N}$. Assume that $\{\rho_{\varepsilon}\}_{\varepsilon > 0}$
is a family of functions satisfying \eqref{AssRho1} and \eqref{AssRho}.
Then, for any $f\in \dot{W}^kX$,
\begin{equation}\label{Equiv1}
\lim\limits_{\varepsilon\to0^+}
\int_{0}^{\infty}\frac{1}{t^{\frac{N}{k}+p}}
\left\|\left[
\int_{|h|\le t^{\frac{1}{k}}}\left|\Delta^k_hf(\cdot)\right|^p\,dh
\right]^{\frac{1}{p}}\right\|_{X}^p\rho_{\varepsilon}(t)\,dt =\frac{1}{kp + N}
\left\|\left[\int_{\mathbb{S}^{N-1}}
\left|\sum_{
|\alpha|=k}\omega^\alpha\partial^{\alpha}
f(\cdot)\right|^p\,d\sigma^{N-1}(\omega)
\right]^{\frac{1}{p}}\right\|_X^p.
\end{equation}
In particular,
\begin{equation}\label{Equiv2}
	\lim\limits_{\varepsilon\to0^+}
\int_{0}^{\infty}\frac{1}{t^{\frac{N}{k}+p}}
\left\|\left[
\int_{|h|\le t^{\frac{1}{k}}}\left|\Delta^k_hf(\cdot)\right|^p\,dh
\right]^{\frac{1}{p}}\right\|_{X}^p\rho_{\varepsilon}(t)\,dt  \approx \|\nabla^k f\|_X^p.
\end{equation}
\end{theorem}

\begin{proof}
Let 	$$
		g(t): = \left\|  \bigg(\frac{1}{t^{\frac{N}{k}}} \,
\int_{|h| \leq t^{\frac{1}{k}}} \left|\Delta_ h f (\cdot)\right|^p \,
 dh \bigg)^{\frac{1}{p}} \right\|_X, \qquad t > 0.
	$$
	We focus on $k=1$, since the general case $k \in \mathbb{N}$ can be reduced to $k=1$ via some technical modifications as  already used in the proof of Theorem \ref{ThmGenSobk}.
	
	It follows from \eqref{Minkow} that
	$$
		g(t) \lesssim \bigg( \frac{1}{t^N} \int_{|h| \leq t}
\left\|\Delta_h f\right\|_X^p \, dh \bigg)^{\frac{1}{p}},
	$$
	which leads to $g(t) \lesssim t  \, \|\nabla f\|_X.$
	Hence \eqref{LimAs1} holds.
	
	We claim that
	\begin{equation}\label{ClaimNewV}
		\lim_{t \to 0^+} \left[\frac{g(t)}{t}
\right]^p = \frac{1}{p+N} \,
\left\|\left[\int_{\mathbb{S}^{N-1}}\left|\omega\cdot
\nabla f(\cdot)\right|^p\,d\sigma^{N-1}(\omega)
\right]^{\frac{1}{p}}\right\|_X^p = \frac{C_{N,p}}{p+N} \, \left\|\nabla f\right\|_X^p,
	\end{equation}
where $C_{N, p}$ is given by \eqref{CBBM}. Assuming this momentarily, one can invoke Lemma \ref{LemmaReal} to obtain
		\begin{equation}\label{BBMMixed}
		\lim_{\varepsilon \to 0^+} \int_0^\infty \frac{1}{t^{N +p}} \,  \bigg\| \int_{|h| \leq t} |\Delta_ h f (\cdot)|^p \, dh  \bigg\|_{X^{1/p}} \rho_\varepsilon(t) \, dt  = \frac{C_{N, p}}{p+N} \, \|\nabla f\|^p_X,
	\end{equation}
completing the proof of the formula \eqref{Equiv1}. 
	
Next, we show \eqref{ClaimNewV}.
Assume first that $f\in C^\infty(\R^N)$,
$\nabla f$ is compactly supported in $\Omega$,
and $\nabla f$ is Lipschitz with a constant $A$. Then
	$$
		|f(x+h)-f(x)-h \cdot \nabla f (x)| \leq A \, |h|^2 \, \mathbf{1}_{2 \Omega}(x)
	$$
	for $h$ sufficiently small (depending on $f$).
	In particular, given any $\theta >0$ there exists $C_\theta > 0$ such that
	$$
		|f(x+h)-f(x)|^p \leq (1+\theta) \, |h \cdot \nabla f(x)|^p + C_\theta A^p |h|^{2 p}  \mathbf{1}_{2 \Omega}(x).
	$$
	For $t$ sufficiently small, using polar coordinates, one has
	\begin{align*}
		  \int_{|h| \leq t} |\Delta_ h f (x)|^p \, dh &\leq (1+\theta) \, \int_{|h| \leq t}  |h \cdot \nabla f(x)|^p  \, dh + C_\theta A^p \int_{|h| \leq t} |h|^{2 p} \, dh \, \mathbf{1}_{2 \Omega}(x) \\
		  & = C_{N, p} (1+\theta) \frac{t^{N+p}}{N+p} \, |\nabla f(x)|^p  + C_\theta A^p \frac{t^{2 p + N}}{2 p + N} \, \mathbf{1}_{2 \Omega}(x),
	\end{align*}
	and thus
	$$
		\bigg\|\bigg(\frac{1}{t^{N+p}} \int_{|h| \leq t} |\Delta_ h f (x)|^p \, dh \bigg)^{\frac{1}{p}} \bigg\|_X \leq \bigg\|\bigg(\frac{C_{N, p}}{N+p} (1+\theta)  \, |\nabla f(x)|^p +   C_\theta A^p \frac{t^{p}}{2 p + N} \, \mathbf{1}_{2 \Omega}(x) \bigg)^{\frac{1}{p}} \bigg\|_X.
	$$
	Taking limits as $t \to 0^+$ on both sides of the previous inequality and applying dominated convergence theorem (we  assume that $X$ has absolutely continuous norm, see \cite[Proposition 3.6, p. 16]{BennettSharpley}), we obtain
	$$
		\limsup_{t \to 0^+} \, \bigg\|\bigg(\frac{1}{t^{N+p}} \int_{|h| \leq t} |\Delta_ h f (x)|^p \, d h \bigg)^{\frac{1}{p}} \bigg\|_X \leq  \bigg(\frac{C_{N, p}}{N+p} \bigg)^{\frac{1}{p}} (1+\theta)^{\frac{1}{p}}  \, \|\nabla f\|_X.
	$$
	Since this holds for arbitrary $\theta > 0$, we conclude that
	$$
		\limsup_{t \to 0^+} \, \bigg\|\bigg(\frac{1}{t^{N+p}} \int_{|h| \leq t} |\Delta_ h f (x)|^p \, d h \bigg)^{\frac{1}{p}} \bigg\|_X \leq  \bigg(\frac{C_{N, p}}{N+p} \bigg)^{\frac{1}{p}}  \, \|\nabla f\|_X.
	$$
	The converse inequality, i.e.,
	$$
		\liminf_{t \to 0^+} \, \bigg\|\bigg(\frac{1}{t^{N+p}} \int_{|h| \leq t} |\Delta_ h f (x)|^p \, d h \bigg)^{\frac{1}{p}} \bigg\|_X \geq  \bigg(\frac{C_{N, p}}{N+p} \bigg)^{\frac{1}{p}}  \, \|\nabla f\|_X
	$$
	can be obtained in a similar fashion. Thus \eqref{ClaimNewV} is shown for any
$f\in C^\infty(\R^N)$ with $\nabla f\in C^\infty_{{\rm c}}(\R^N,\R^N)$. The general case $f\in \dot{W}^kX$ follows from standard density arguments.

{To show \eqref{Equiv2}, since \eqref{Equiv1} was already established,  it is enough to see that
\begin{equation}\label{Equiv3}
\int_{\mathbb{S}^{N-1}}
\left|\sum_{
|\alpha|=k}\omega^\alpha\partial^{\alpha}
f(\cdot)\right|^p\,d\sigma^{N-1}(\omega) \approx |\nabla^k f(x)|^p.
\end{equation}
Indeed, for any   $v = (v_\alpha)_{|\alpha| = k} \in \R^{\mathcal{K}}$, where $\mathcal{K}$ is the cardinality of the set $\{\alpha \in \Z_+^N : |\alpha| = k\}$, define
$$
	E(v) := \bigg(\int_{\mathbb{S}^{N-1}} |W \cdot v|^p \, d \sigma^{N-1}(\omega) \bigg)^{\frac{1}{p}},
$$
where $W := (\omega^\alpha)_{|\alpha| = k}$. Note that $E$ defines a norm $\R^{\mathcal{K}}$. In particular, $E(v) \approx |v|$, the Euclidean norm on $\R^{\mathcal{K}}$. This shows the desired equivalence \eqref{Equiv3} by taking $v = (\partial^\alpha f(x))_{|\alpha| = k}$. }
\end{proof}

\begin{remark}
Some examples of spaces $X$ satisfying \eqref{Minkow} are $L^q, \, q \in [p, \infty],$ and the Lorentz spaces $L^{q, r}$ provided that $q \in (p, \infty)$ and  $r \in [p, \infty]$. In particular, letting $X=L^p, \, p \in [1, \infty)$, in Theorem \ref{ThmGenSobk'} and applying Fubini's theorem, we arrive at
	\begin{equation}\label{BBMMixed2}
		\lim_{\varepsilon \to 0^+} \int_{\R^N} \int_{\R^N} |f(x)-f(y)|^p \, \phi_\varepsilon(|x-y|) \, dx \, dy = \frac{C_{N, p}}{N+p} \, \int_{\R^N} |\nabla f (x)|^p \, dx.
	\end{equation}
Comparing this statement with the classical BBM formula \eqref{BBM}, the additional constant $(N+p)^{-1}$ arising in \eqref{BBMMixed2} is compensated by the change from the standard weight $\frac{\rho_\varepsilon}{|x-y|^p}$ in \eqref{BBM} to $\phi_\varepsilon$.
\end{remark}


The next result refers to Maz'ya--Shaposhnikova-type characterizations of \emph{ball Banach function spaces}, see \cite[Definition 2.1]{DGPYYZ} for the precise definition. Examples of these spaces include Morrey spaces, mixed Lebesgue spaces, weighted Lebesgue spaces, Orlicz spaces,  and Lorentz spaces.

\begin{theorem}\label{ThmGenSobMSk'}
Let $p\in(0,\infty)$, $k\in\mathbb{N}$, and let
$X$ be a ball Banach function space on $\R^N$. Assume that
$\{\psi_{\varepsilon}\}_{\varepsilon\in(0,\infty)}$ is
a family of functions satisfying \eqref{AssPsi1}
and \eqref{AssPsi}.  Assume that the Hardy--Littlewood maximal operator $\mathcal{M}$
  is bounded on $X^{\frac{1}{p}}$. Then, for any $f \in L^\infty_{\rm c}(\R^N)$,
  \begin{equation}\label{ThmGenSobMSk'e4}
  \lim_{\varepsilon\to0^+}
\int_{0}^{\infty}\left\|
\left[\frac{1}{t^{\frac{N}{k}}}
\int_{|h|\le t^{\frac{1}{k}}}
\left|\Delta^k_hf(\cdot)\right|^p\,dh\right]^{\frac{1}{p}}
\right\|_{X}^p\psi_{\varepsilon}(t)\,dt
=\frac{\left|\mathbb{S}^{N-1}\right|}{N}  \, \|f\|_X^p.
\end{equation}
If, in addition, $X$ has absolutely continuous norm then \eqref{ThmGenSobMSk'e4} holds for all $f \in X$.

\end{theorem}

\begin{rem}
	The assumption that $\mathcal{M}$ acts boundedly on $X^{\frac{1}{p}}$ is in fact necessary in Theorem \ref{ThmGenSobMSk'}. Indeed,  suppose that \eqref{ThmGenSobMSk'e4} holds with $X = L^p(\R^N)$ (of course then our condition on
   $\mathcal{M}$ is not satisfied since $X^{1/p} = L^1(\R^N)$). This says (after applying Fubini's theorem) that
$$
	  \lim_{\varepsilon\to0^+}   \int_{\R^N}
\|\Delta_hf \|_{L^p(\R^N)}^p  \,
\varphi_\varepsilon(|h|) \,dh
=\frac{\left|\mathbb{S}^{N-1}\right|}{N}  \, \|f\|_{L^p(\R^N)}^p,
$$
where $\varphi_\varepsilon$ is given by \eqref{612new}. However,
this statement contradicts \eqref{ThmGenSobMSk1} with \eqref{LpExDif}, i.e.,
$$
	  \lim_{\varepsilon\to0^+}   \int_{\R^N}
\|\Delta_hf \|_{L^p(\R^N)}^p  \,
\varphi_\varepsilon(|h|) \,dh
=2 \, \frac{\left|\mathbb{S}^{N-1}\right|}{N}  \, \|f\|_{L^p(\R^N)}^p.
$$
This observation confirms that finding  the limit \eqref{ThmGenSobMSk'e4} is a rather delicate question, since its exact value depends on additional properties of the involved space  $X$ (even in the r.i. setting). This is in sharp contrast with the counterpart stated in \eqref{ThmGenSobMSk1}, which holds for any r.i. space $X$.
\end{rem}

\begin{proof}[Proof of Theorem \ref{ThmGenSobMSk'}]

To show \eqref{ThmGenSobMSk'e4}, we are going to apply Lemma \ref{LemmaReal}(ii) to the function
\begin{equation*}
g(t):=\left\|
\left[\frac{1}{t^{\frac{N}{k}}}
\int_{|h|\le t^{\frac{1}{k}}}
\left|\Delta^k_hf(\cdot)\right|^p\,dh\right]^{\frac{1}{p}}
\right\|_{X}, \qquad t > 0.
\end{equation*}

From the boundedness of $\mathcal{M}$ on $X^{1/p}$, we derive
\begin{equation}\label{thmgensobmsk'e1}
g(t)
\lesssim\sum_{j=0}^{k}
\left\|\frac{1}{t^{\frac{N}{k}}}
\int_{|h|\le t^{\frac{1}{k}}}
\left|f(\cdot+jh)\right|^p\,dh
\right\|_{X^{\frac{1}{p}}}^{\frac{1}{p}} \lesssim \|f\|_{X}+\sum_{j=1}^{k} j^{\frac{N}{p}} \left\|
\mathcal{M}(|f|^p)\right\|_{X^{\frac{1}{p}}}
^{\frac{1}{p}}
\lesssim \|f\|_X,
\end{equation}
i.e., $g$ satisfies \eqref{LimAspsi1}.

Assume momentarily that
\begin{equation}\label{ThmGenSobMSk'e5}
\lim_{t\to\infty} g(t)^p=
\frac{\left|\mathbb{S}^{N-1}\right|}{N} \, \|f\|_X^p.
\end{equation}
Thus the desired result \eqref{ThmGenSobMSk'e4} follows from Lemma \ref{LemmaReal}(ii) with \eqref{thmgensobmsk'e1} and \eqref{ThmGenSobMSk'e5}.

Next we show the claim \eqref{ThmGenSobMSk'e5}. We first assume that $f\in L^\infty_{{\rm c}}(\R^N)$, say, $\supp f \subset B(0,R)$ for some $R > 0$.
Then, repeating an argument similar to that used in \eqref{thmgensobmsk'e1} with $t^{1/k}$ replaced
by $2R$, we have
\begin{align}\label{ThmGenSobMSk'e6}
I_0(t):&=
\left\|\left[ \frac{1}{t^{\frac{N}{k}}} \int_{|h|\le 2R}\left|\Delta^k_hf(\cdot)\right|^p
\,dh \right]^{\frac{1}{p}}\right\|_{X}\\
&\lesssim t^{-\frac{N}{k p}}\left[\|f\|_X+(2R)^\frac{N}{p}\sum_{j=1}^{k} j^{\frac{N}{p}}
\left\|\mathcal{M}(|f|^p)\right\|_{X^{\frac{1}{p}}}^{\frac{1}{p}} \right]\notag\\
&\lesssim t^{-\frac{N}{k p}}\left[\|f\|_X+
\left\|\mathcal{M}(|f|^p)\right\|_{X^{\frac{1}{p}}}^{\frac{1}{p}}\right]\lesssim
t^{-\frac{N}{k p}}\|f\|_X\to0\notag
\end{align}
as $t\to\infty$.
In addition, we claim that, for any $j\in\{1,\ldots,k\}$,
\begin{align}\label{ThmGenSobMSk'e7}
I_j(t):=
\left\|\left[ \frac{1}{t^{\frac{N}{k}}} \int_{2R<|h|\le t^{\frac{1}{k}}}
\left|f(\cdot+jh)\right|^p\,dh \right]^{\frac{1}{p}}\right\|_{X}
\to0
\end{align}
as $t\to\infty$. Indeed,  if
$t>(2R)^k$ and $x \in B(0, R) \cup B(0, R + k t^{\frac{1}{k}})^c$ then
$$
	f(x + j h) = 0 \qquad \text{for} \qquad 2 R < |h| \leq t^{\frac{1}{k}}.
$$
Accordingly (noting that $f \in L^p(\R^N)$ since $f \in L^\infty_c(\R^N)$)
\begin{align}\label{ThmGenSobMSk'e8}
I_j(t)^p&=\frac{1}{t^{\frac{N}{k}}}
\left\|\int_{2R<|h|\le t^{\frac{1}{k}}}|f(\cdot+jh)|^p\,dh \,
\mathbf{1}_{B(0,R+kt^{\frac{1}{k}})\setminus
B(0,R)} \right\|_{X^{\frac{1}{p}}} \le t^{-\frac{N}{k}}\|f\|_{L^p(\R^N)}^p
\left\|\mathbf{1}_{B(0,(k+1)t^{\frac{1}{k}})}
\right\|_{X^{\frac{1}{p}}}.
\end{align}
From the boundedness of $\mathcal{M}$ on $X^{\frac{1}{p}}$
and \cite[Lemma 2.15(ii)]{shyy17}, we deduce that
there exists an $\eta\in(1,\infty)$ such that\footnote{Recall that $\mathcal{M}^{(\eta)}(f):=[\mathcal{M}(|f|^{\eta})]^{\frac{1}{\eta}}$.}
$\mathcal{M}^{(\eta)}$ is bounded on $X^{\frac{1}{p}}$.
Combining this together with \eqref{ThmGenSobMSk'e8} and  the estimate
 $\mathbf{1}_{B(0,(k+1)t^{\frac{1}{k}})}
\lesssim t^{\frac{N}{k\eta}}\mathcal{M}^{(\eta)}(\mathbf{1}
_{B(0,1)})$, we find
\begin{align*}
I_j(t)^p\lesssim t^{\frac{N}{k}(\frac{1}{\eta}-1)}
\left\|\mathcal{M}^{(\eta)}(\mathbf{1}
_{B(0,1)})\right\|_{X^{\frac{1}{p}}}
\lesssim t^{\frac{N}{k}(\frac{1}{\eta}-1)}
\left\|\mathbf{1}_{B(0,1)}\right\|_X^p
\to0
\end{align*}
as $t\to\infty$. This completes  the
proof of \eqref{ThmGenSobMSk'e7}.

Observe that, for any $h\in\R^N$ with $|h|>2R$, the sets
$\{\supp f(\cdot+jh)\}_{j=0}^k$ are pairwise disjoint.
Thus, for any $\theta > 0$ there exists $C_\theta > 0$ such that, for $t>(2R)^k$,
\begin{align*}
g(t)
&\le C_\theta
\left\| \left[\frac{1}{t^{\frac{N}{k}}} \int_{|h|\le 2R}
\left|\Delta^k_hf(\cdot)\right|^p\,dh \right]^{\frac{1}{p}}\right\|_X  + (1+\theta)\left\| \left[\frac{1}{t^{\frac{N}{k}}}  \int_{2R<|h|\le t^{\frac{1}{k}}}
\left|\sum_{j=0}^{k}(-1)^{k-j}\binom{k}{j} f(\cdot+jh) \right|^p
\,dh \right]^{\frac{1}{p}}\right\|_{X} \\
&=C_\theta I_0(t)+ (1+\theta)
\left\| \left[\frac{\left|\mathbb{S}^{N-1}\right|}{N}\left[1-(2R)^N t^{-\frac{N}{k}}\right]|f|^p
+\sum_{j=1}^{k}\binom{k}{j}^p \frac{1}{t^{\frac{N}{k}}} \int_{2R<|h|\le t^{\frac{1}{k}}}
|f(\cdot+jh)|^p\,dh \right]^{\frac{1}{p}}\right\|_X\\
&\le C_\theta I_0(t) +(1+\theta) C_\theta 2^{\min\{0, \frac{1}{p}-1\} k} \sum_{j=1}^{k}\binom{k}{j} I_j(t)+ (1+\theta)^2
\bigg(\frac{\left|\mathbb{S}^{N-1}\right|}{N} \bigg)^{\frac{1}{p}}
\left[1-(2R)^Nt^{-\frac{N}{k}}\right]^{\frac{1}{p}}\|f\|_X.
\end{align*}
Applying \eqref{ThmGenSobMSk'e6} and \eqref{ThmGenSobMSk'e7}
and letting $t\to\infty$ and then $\theta \to 0^+$,
we  obtain
\begin{equation}\label{Limsupnew}
\limsup_{t\to\infty} g(t)\le \bigg(
\frac{\left|\mathbb{S}^{N-1}\right|}{N} \bigg)^{\frac{1}{p}}\|f\|_X.
\end{equation}

Conversely, for any $h\in\R^N$ with $|h|>2R$,
by the disjointness of
$\{\supp f(\cdot+jh)\}_{j=0}^k$ again, we conclude that,
for any $t>(2R)^k$,
\begin{equation}\label{ThmGenSobMSk'e10}
g(t)^p
\ge\frac{1}{t^{\frac{N}{k}}}
\left\|\int_{2R<|h|\le t^{\frac{1}{k}}}
\left|\Delta^k_hf(\cdot)\right|^p\,dh\right\|_{X^{\frac{1}{p}}} \ge \frac{\left|\mathbb{S}^{N-1}\right|}{N}
\left[1-(2R)^Nt^{-\frac{N}{k}}\right]\|f\|_X^p.
\end{equation}
Letting $t\to\infty$ in \eqref{ThmGenSobMSk'e10},
we  have
\begin{equation}\label{Liminfnew}
\liminf_{t\to\infty} g(t)^p\ge
\frac{\left|\mathbb{S}^{N-1}\right|}{N} \,
\|f\|_X^p.
\end{equation}

As a combination of \eqref{Limsupnew} and \eqref{Liminfnew}, we conclude that \eqref{ThmGenSobMSk'e5} holds for any $f\in L^\infty_{\rm c}(\R^N)$.
If, in addition, $X$ has an absolutely continuous norm, then standard density arguments can be applied to show  that \eqref{ThmGenSobMSk'e5} holds
for any $f\in X$.
\end{proof}

\subsection{Limiting formulas for fractional powers of generators of semigroups}\label{Section6.4}

In this section, we deal with a Banach space $X$ and a  \emph{strongly continuous semigroup} of operators $\{T_t\}_{t > 0}$ acting on $X$, i.e.,  $T_t : X \to X$ with
\begin{equation}\label{SG}
	T_{t+\xi} = T_t T_\xi, \qquad \lim_{t \to 0^+} T_t f = f \quad \text{for} \quad  f \in X, \qquad \text{and} \qquad \|T_t\|_{X \to X} \leq M.
\end{equation}
The corresponding \emph{infinitesimal generator} $\mathcal{A}$ is defined (at least, formally) via the strong limit
\begin{equation}\label{Gen}
	\mathcal{A} f := \lim_{t \to 0^+} \frac{[T_t  - I] f}{t}, \qquad \mathcal{A}^k = \mathcal{A} (\mathcal{A}^{k-1}), \quad k \in \N.
\end{equation}
The set of all elements $f \in X$ such that this limit exists is the domain $D(\mathcal{A})$, which is endowed with
$
	\|f\|_{D(\mathcal{A})} := \|\mathcal{A} f\|_X.
$

Following \cite{Westphal}, \eqref{Gen} can be naturally extended to fractional powers of order $\alpha > 0$ of $-\mathcal{A}$ by\footnote{Recall the well-known fact that $(-\mathcal{A})^\alpha f = C_{\alpha, n} \, \int_0^\infty \frac{[I-T_u]^n f}{u^{\alpha + 1}} \, du$, where $0 < \alpha < n \in \N$.}
\begin{equation}\label{DefFG}
	(-\mathcal{A})^\alpha f := \lim_{t \to 0^+} \frac{[I-T_t]^\alpha f}{t^\alpha}
\end{equation}
whenever this limit exists, and $D((-\mathcal{A})^\alpha)$ is the set of all $f \in X$ such that \eqref{DefFG} is finite and let $\|f\|_{D((-\mathcal{A})^\alpha)} := \|(-\mathcal{A})^\alpha f\|_X$. Here
\begin{equation}\label{DefFG2}
	[I-T_t]^\alpha := \sum_{j=0}^\infty (-1)^j \binom{\alpha}{j} T_{j t}, \qquad \binom{\alpha}{j} := \prod_{l=1}^j \frac{\alpha-l + 1}{l}\quad  \text{if} \quad   j \in \N,  \quad \binom{\alpha}{0} := 1.
\end{equation}
In the integer case $\alpha = k$, both definitions of $(-\mathcal{A})^k$ (cf.  \eqref{Gen} and \eqref{DefFG}) coincide.  

Distinguished examples of strongly continuous semigroups on $L^p(\R^N)$ are translations (if $N=1$ and $\mathcal{A} f = f'$), heat semigroup (with $\mathcal{A} = \Delta$, the \emph{Laplace operator}), Poisson semigroup (with $\mathcal{A}^2 = \Delta$), Ornstein--Uhlenbeck semigroup ($\mathcal{A} = \Delta - x \cdot \nabla$), as well as semigroups related to Kolmogorov operators or, more generally, Kolmogorov--Fokker--Planck operators related to  H\"ormander semigroups \cite{Hormander}. On the other hand, the well-known examples of semigroups on $L^p(\R^N, \gamma_N)$ include Ornstein--Uhlenbeck and Poisson--Hermite semigroups. On $L^p(\mathbb{S}^{N-1})$ (equipped with the surface measure), one may consider e.g. the Weierstrass semigroup in terms of spherical harmonic whose related generator is the Laplace--Beltrami operator.

\begin{thm}\label{ThmSemGen}
Let $p, \alpha > 0$.
\begin{enumerate}[\upshape(i)]
	\item Given  a family of functions $\{\rho_\varepsilon\}_{\varepsilon > 0}$ satisfying \eqref{AssRho1} and  \eqref{AssRho},
	then
\begin{equation*}
	\lim_{\varepsilon \to 0^+} \, \left[ \int_0^\infty  \frac{\|[I-T_t]^\alpha f\|^p_X}{t^{\alpha(p-1)}} \,  \rho_\varepsilon (t^\alpha) \, \frac{dt}{t} \right]^{\frac{1}{p}} = \frac{1}{\alpha^{\frac{1}{p}}} \,  \|(-\mathcal{A})^\alpha f\|_{X}
\end{equation*}
for any $f \in D((-\mathcal{A})^\alpha)$.
\item  Assume that the limit
$
	\lim_{t \to \infty} T_t f
$
exists for all $f \in X$ and let
\begin{equation}\label{Pf}
	P f := \lim_{t \to \infty} T_t f.
\end{equation}
Given  a family of functions $\{\psi_\varepsilon\}_{\varepsilon > 0}$ satisfying  \eqref{AssPsi1} and \eqref{AssPsi},
	then
\begin{equation*}
	\lim_{\varepsilon \to 0^+} \, \left( \int_0^\infty \|[I-T_t]^\alpha f\|_X^p \,  \psi_\varepsilon (t) \, dt \right)^{\frac{1}{p}} = \| f + c_\alpha P f\|_{X}
\end{equation*}
for any $f \in X$. Here $c_\alpha := \sum_{j=1}^\infty (-1)^j \binom{\alpha}{j}$.
\end{enumerate}
\end{thm}

\begin{rem}
\begin{enumerate}[\upshape(i)]
	\item Characterizations of the existence of the limit \eqref{Pf} are known, see \cite{Bobrowski}. For example, this is the case for  an analytic semigroup $\{T_t\}_{t > 0}$ on a reflexive Banach space $X$. See also \cite{Garofalo}, where rates of convergence are provided in the special case of H\"ormander semigroups on $X = L^p(\R^N)$.
	\item Different approaches to limiting formulas based on semigroups in the classical setting $\alpha =1$ may be found in \cite{Garofalo} and \cite{Milman}.
\end{enumerate}
\end{rem}

\begin{proof}[Proof of Theorem \ref{ThmSemGen}]
(i): Let
$$
	g(t^\alpha) := \|[I-T_t]^\alpha f\|_X, \qquad \forall t > 0.
$$
According to \cite[(2.12), p. 563]{Westphal}, the following formula holds
$$
	[I-T_t]^\alpha f = t^\alpha \int_0^\infty p_\alpha \Big(\frac{u}{t} \Big) T_u (-\mathcal{A})^\alpha f \, \frac{du}{t},
$$
where the function $p_\alpha$ was explicitly computed in \cite[p. 560]{Westphal}. In particular, $\int_0^\infty p_\alpha(u) \, du = 1$. Hence
$$
	g(t^\alpha) \leq t^\alpha \int_0^\infty p_\alpha \Big(\frac{u}{t} \Big) \|T_u (-\mathcal{A})^\alpha f \|_X \, \frac{du}{t} \leq M t^\alpha  \| f \|_{D((-\mathcal{A})^\alpha)},
$$
which implies \eqref{LimAs1}. Furthermore, by \eqref{DefFG},
$$
	\lim_{t \to 0^+} \frac{g(t^\alpha)}{t^\alpha} = \|(-\mathcal{A})^\alpha f\|_X.
$$
We are in a position to apply Lemma \ref{LemmaReal}(i) and get
	\begin{equation*}
	\alpha^{\frac{1}{p}}  	\bigg[\int_0^\infty \frac{\|[I-T_t]^\alpha f\|^p_X}{t^{\alpha (p-1) }} \,  \rho_\varepsilon(t^\alpha)  \, \frac{dt}{t} \bigg]^{\frac{1}{p}} = \bigg[\int_0^\infty \bigg(\frac{g(t)}{t} \bigg)^p \rho_\varepsilon(t) \, dt \bigg]^{\frac{1}{p}}  \to  \|(-\mathcal{A})^\alpha f\|_X
	\end{equation*}
	as $\varepsilon \to 0^+$.

(ii): By \eqref{DefFG2} and \eqref{SG}, we have
$$
	g(t^\alpha) \leq \sum_{j=0}^\infty \bigg| \binom{\alpha}{j}  \bigg| \,  \|T_{j t} f \|_X \leq M \, \|f\|_X \sum_{j=0}^\infty \bigg| \binom{\alpha}{j}  \bigg|,
$$
where the last series is convergent. Furthermore, it is clear that
$$
	\lim_{t \to \infty} \, [I-T_t]^\alpha f = f + c_\alpha \, P,
$$
and thus $\lim_{t \to \infty} g(t^\alpha) = \|f + c_\alpha P\|_X$. We can invoke Lemma \ref{LemmaReal}(ii) to get
$$
		\lim_{\varepsilon \to 0^+} \,
\bigg[\int_0^\infty \|[I-T_t]^\alpha f\|_X^p \,
 \psi_\varepsilon(t) \, dt \bigg]^{\frac{1}{p}} = \| f + c_\alpha \, P\|_X.
$$
\end{proof}

We now mention several important corollaries  of the previous result obtained using the functions from Examples \ref{Ex1} and \ref{Ex1Dual} in Appendix \ref{SectionAA}. In more detail, we consider $\rho_\varepsilon(t) = \varepsilon t^{\varepsilon} \mathbf{1}_{(0, 1)}(t)$ (with $\varepsilon = \frac{p}{\alpha} (\alpha - s)$) and $\psi_\varepsilon(t) = \varepsilon t^{\varepsilon-1} \mathbf{1}_{(1, \infty)}(t)$ (with $\varepsilon = s p$) to get Corollary \ref{6.58}, $\rho_\varepsilon(t) = \frac{\beta}{\varepsilon^{\alpha \beta}} t^\beta \mathbf{1}_{(0, \varepsilon^\alpha)}(t)$ and $\psi_{\varepsilon}(t)=
\frac{\beta}{\varepsilon^\beta t^{\beta+1}}{\bf 1}
_{(\varepsilon^{-1},\infty)}(t)$ to get  Corollary \ref{6.60}, $\rho_\varepsilon (t) = \frac{1}{|\log \varepsilon| t} \, \mathbf{1}_{(\varepsilon, 1)}(t)$ and $\psi_\varepsilon (t) = \frac{1}{|\log \varepsilon| t} \, \mathbf{1}_{(1, \varepsilon^{-1})}(t)$ to get  Corollary \ref{6.61}.

\begin{cor}\label{6.58}
Let $p, \alpha \in (0, \infty)$. Then, for $f \in X \cap D((-\mathcal{A})^\alpha)$,
$$
	\lim_{s \to \alpha^-} \, (\alpha-s) \, \int_0^\infty \frac{\|[I-T_t]^\alpha f\|_X^p}{t^{s p}} \, \frac{dt}{t} = \frac{1}{p} \, \|(-\mathcal{A})^\alpha f\|_X^p
$$
and
	$$
		\lim_{s \to 0^+} \, s \, \int_0^\infty \frac{\|[I-T_t]^\alpha f\|_X^p}{t^{s p}} \, \frac{dt}{t} = \frac{1}{p} \, \|f + c_\alpha P f\|_X^p.
	$$
\end{cor}

\begin{rem}\label{Remark659}
	 In the special case $X = L^p(\R), \, p \in [1, \infty),$ and $T_t f(x) = f(x+t)$, our results recover classical Bourgain--Brezis--Mironescu and Maz'ya--Shaposhnikova formulas if $\alpha =1$, as well as the recent results given in \cite{DominguezMilman} if $\alpha > 0$. Note that $\mathcal{A} f =f'$ and $(-\mathcal{A})^\alpha f = \lim_{t \to 0^+} \frac{\Delta_t^\alpha f}{t^\alpha}$, the \emph{Liouville--Gr\"unwald--Letnikov derivative} of order $\alpha$, where $\Delta^\alpha_t$ is the fractional difference operator given by \eqref{DefFG2}; see e.g. \cite[§20, (20.7)]{Samko}. Recall that $\|(-\mathcal{A})^\alpha f\|_{L^p(\R)} \approx \|(-\Delta)^{\alpha/2} f\|_{L^p(\R)}$ if $p \in (1, \infty)$.
\end{rem}

\begin{cor}\label{6.60}
Let $p, \alpha, \beta \in(0,\infty)$.
Then, for $f \in D((-\mathcal{A})^\alpha)$,
\begin{align*}
\lim\limits_{\varepsilon\to0^+}
\frac{1}{\varepsilon^{\alpha \beta}}
\int_{0}^{\varepsilon} \frac{\left\|[I-T_t]^\alpha f\right\|_{X}^p}{t^{\alpha(p-\beta)}} \,\frac{dt}{t}
= \frac{1}{\alpha \beta} \, \left\|(-\mathcal{A)^\alpha}f\right\|^p_X,
\end{align*}
and, for $f \in X$,
\begin{align*}
\lim\limits_{\varepsilon\to\infty}
\varepsilon^\beta \int_{\varepsilon}^{\infty}
\frac{\left\|[I-T_t]^\alpha f\right\|_{X}^p}{t^{\beta}}
\,\frac{dt}{t}
= \frac{1}{\beta} \,  \|f + c_\alpha P f\|^p_X.
\end{align*}
\end{cor}


\begin{cor}\label{6.61}
Let $p, \alpha \in(0,\infty)$.
Then, for $f \in X \cap D((-\mathcal{A})^\alpha)$,
\begin{align*}
\lim\limits_{\varepsilon\to0^+}
\frac{1}{|\log\varepsilon|}
\int_{\varepsilon}^{\infty}
\frac{\|[I-T_t]^\alpha f\|_X^p}{t^{\alpha p}}\,\frac{dt}{t}
=\left\|(-\mathcal{A})^\alpha f\right\|_X^p
\end{align*}
and
\begin{align*}
\lim\limits_{\varepsilon\to\infty}
\frac{1}{|\log\varepsilon|}
\int_{0}^{\varepsilon}
\|[I-T_t]^\alpha f\|_X^p\,\frac{dt}{t} = \|f + c_\alpha P f\|^p_X.
\end{align*}
\end{cor}

\appendix

\section{Examples of operators approximating unity}\label{SectionAA}

	\begin{ex}\label{Ex1}
	\begin{enumerate}[\upshape(i)]
	\item $\rho_\varepsilon(t) = \varepsilon  t^{\varepsilon -1} \, \mathbf{1}_{(0, 1)}(t)$.
	\item Assume that $\xi$ is a compactly supported positive function with $\int_{0}^\infty \xi (t) \, dt = 1$ and let
	$$
		\rho_\varepsilon (t) = \frac{1}{\varepsilon} \, \xi \bigg(\frac{t}{\varepsilon}\bigg).
	$$
	The prototypical example is given by
	\begin{equation}\label{27new}
	\xi(t) = \alpha t^{\alpha-1} \, \mathbf{1}_{(0, 1)}(t), \qquad \alpha > 0.
	\end{equation}
	\item 	$\rho_\varepsilon (t) = \frac{1}{|\log \varepsilon| t} \, \mathbf{1}_{(\varepsilon, 1)}(t)$.
	\end{enumerate}
	
	\end{ex}

	Examples of families $\{\psi_\varepsilon\}$ satisfying \eqref{AssPsi1}-\eqref{AssPsi} can be easily obtained from $\{\rho_\varepsilon\}$ with  \eqref{AssRho1}-\eqref{AssRho}  via the transformation
	$$
	\psi_\varepsilon (t) = \frac{1}{t^2} \, \rho_\varepsilon \Big(\frac{1}{t} \Big).
	$$

		\begin{ex}\label{Ex1Dual}
\begin{enumerate}[\upshape(i)]
\item 	$\psi_\varepsilon(t) = \varepsilon \, t^{-\varepsilon -1} \, \mathbf{1}_{(1, \infty)}(t)$.	
\item $\psi_{\varepsilon}(t)=
\frac{\alpha}{\varepsilon^\alpha t^{\alpha+1}}{\bf 1}
_{(\varepsilon^{-1},\infty)}(t), \, \alpha > 0$.
\item $\psi_\varepsilon (t) = \frac{1}{|\log \varepsilon| t} \, \mathbf{1}_{(1, \varepsilon^{-1})}(t)$.
\end{enumerate}
\end{ex}

\section{Computability of the $K$-functional}\label{AppendixB}

We obtain the following description of the $K$-functional for the pair $(X, \dot{W}^k X)$.

\begin{lem}\label{lemmaAk}
Let $X$ be a r.i. space on $\R^N$. Let $k\in\mathbb{N}$ and $p\in(0,\infty)$. Then
\begin{align*}
K\left(t^k,f;X,\dot{W}^kX\right)
\approx \left(\frac{1}{t^N}
\int_{|h|\le t}\left\|
\Delta^k_hf\right\|_X^p\,dh\right)^{\frac{1}{p}}
\end{align*}
for any $f\in W^kX$ and $t\in(0,\infty)$.
\end{lem}

\begin{proof}
We will make use of the classical estimate (cf. \eqref{EstimKFunctMod})
\begin{equation}\label{lemmaAke0}
K\left(t^k,f;X,\dot{W}^kX\right)
\approx \omega_k(f,t)_X.
\end{equation}
Thus, we only need to prove that
\begin{align}\label{lemmaAke3}
\omega_k(f,t)_X\approx
\left(\frac{1}{t^N}
\int_{|h|\le t}\left\|
\Delta^k_hf\right\|_X^p\,dh\right)^{\frac{1}{p}}.
\end{align}
Indeed, the estimate $\gtrsim$ is obvious.
To deal with the converse estimate, we first claim that
\begin{align}\label{lemmaAke1}
\left\|\Delta^k_{2h}f\right\|_{X}
\lesssim\left\|\Delta^k_hf\right\|_X.
\end{align}
Indeed, applying the translation invariance of $X$, we find that,
for any $h\in\R^N$,
\begin{align*}
\left\|\Delta^k_{2h}f\right\|_X
&=\left\|\sum_{i=0}^{k}\binom{k}{i}\Delta^k_hf(\cdot+ih)
\right\|_X \le\sum_{i=0}^{k}\binom{k}{i}
\left\|\Delta^k_hf(\cdot+ih)\right\|_X=2^k
\left\|\Delta^k_hf\right\|_X,
\end{align*}
which completes the proof
of \eqref{lemmaAke1}.

Observe that, for any $h,x,\xi\in\R^N$,
\begin{align*}
\Delta_h^kf(x)
=\sum_{i=1}^{k}(-1)^i\binom{k}{i}
\left[\Delta^k_{\frac{i(\xi-h)}{k}}
f(x+ih)-\Delta^k_{h+\frac{i(\xi-h)}{k}}f(x)\right].
\end{align*}
From this, the translation invariance
of $X$ again, and \eqref{lemmaAke1},
we infer that, for any $h,\xi\in\R^N$,
\begin{align*}
\left\|\Delta^k_hf\right\|_{X}
&\le\sum_{i=1}^{k}\binom{k}{i}
\left[\left\|\Delta^k_{\frac{i(\xi-h)}{k}}f\right\|_{X}
+\left\|\Delta^k_{h+\frac{i(\xi-h)}{k}}f
\right\|_{X}\right]\\
&\lesssim\sum_{i=1}^{k}
\left[\left\|\Delta^k_{\frac{i(\xi-h)}{2k}}f\right\|_{X}
+\left\|\Delta^k_{h+\frac{i(\xi-h)}{k}}f
\right\|_{X}\right].
\end{align*}
Integrating the previous estimate on $|\xi| \leq t$, we obtain
\begin{align}\label{lemmaAke2}
t^N\left\|\Delta^k_hf\right\|_{X}^p
\lesssim\sum_{i=1}^{k}
\int_{|\xi|\le t}
\left[\left\|\Delta^k_{\frac{i(\xi-h)}{2k}}f\right\|_{X}^p
+\left\|\Delta^k_{h+\frac{i(\xi-h)}{k}}f
\right\|_{X}^p\right]\,d\xi.
\end{align}
For any $t\in(0,\infty)$,
$\xi,h\in\R^N$ with $|\xi|\le t$ and $|h|\le t$,
and $i\in\{1,\ldots,k\}$,
\begin{align*}
\left|\frac{i(\xi-h)}{2k}\right|
\le t\ \qquad \text{and}\qquad \left|h+\frac{i(\xi-h)}{k}\right|\le t.
\end{align*}
Applying this, \eqref{lemmaAke2}, and
a change of variables, we further conclude that,
for any $h\in \R^N$ with $|h|\le t$,
\begin{equation*}
t^N\left\|\Delta^k_hf\right\|_X^p
\lesssim\sum_{i=1}^{k}
\int_{|\xi|\le t}
\left[\frac{i}{2k}\left\|\Delta^k_{\xi}f\right\|_X^p
+\frac{i}{k}\left\|\Delta^k_{\xi} f\right\|_X^p\right]\,d\xi \approx\int_{|\xi|\le t}\left\|\Delta^k_{\xi} f\right\|_X^p\,d\xi.
\end{equation*}
Accordingly the estimate
$\lesssim$ of \eqref{lemmaAke3} holds true,
which further completes the proof of \eqref{lemmaAke3}.
\end{proof}

\end{document}